\documentclass[reqno,10pt]{amsart}

\usepackage[left=3cm,right=3cm,top=3cm,bottom=3cm]{geometry}
\usepackage{amssymb,bm}
\usepackage{mathrsfs}
\usepackage{xcolor}
\usepackage{enumitem}
\usepackage{appendix}
\usepackage{diffcoeff}
\usepackage{amsmath,cases}

\usepackage{algorithm}
\usepackage{algorithmic}

\usepackage{graphicx}
\usepackage{subfig}

\newtheorem{theorem}{Theorem}[section]

\newtheorem{proposition}[theorem]{Proposition}
\newtheorem{corollary}[theorem]{Corollary}

\theoremstyle{definition}

\newtheorem{example}[theorem]{Example}

\theoremstyle{remark}
\newtheorem{remark}[theorem]{Remark}

\numberwithin{equation}{section}

\begin{document}

\title[Statistical conservation laws]{Numerical approximations to statistical conservation laws for scalar hyperbolic equations}



\author[Qian Huang]{Qian Huang*}
\address{Institute of Applied Analysis and Numerical Simulation, University of Stuttgart, 70569 Stuttgart, Germany}
\email{qian.huang@mathematik.uni-stuttgart.de; hqqh91@qq.com}
\thanks{* Corresponding author}

\author[Christian Rohde]{Christian Rohde}
\address{Institute of Applied Analysis and Numerical Simulation, University of Stuttgart, 70569 Stuttgart, Germany}
\email{christian.rohde@mathematik.uni-stuttgart.de}



\date{}

\dedicatory{}

\begin{abstract}
  Motivated by the statistical description of turbulence, we study statistical conservation laws in the form of kinetic-type PDEs for joint probability density functions (PDFs) and cumulative distribution functions (CDFs) associated with solutions of scalar balance laws. Starting from viscous balance laws, the resulting PDF/CDF equations involve \emph{unclosed} conditional averages arising in the viscous terms. We show that these terms exhibit a dissipative anomaly: they remain non-negligible in the vanishing viscosity limit and are essential to preserve the nonnegativity of evolving PDFs. To approximate these PDF/CDF equations in a unified framework, we propose a novel sampling-based estimator for the unclosed terms, constructed from numerical or exact realizations of the underlying balance-law solutions. In certain cases, \emph{a priori} error bounds can be derived, demonstrating that the deviation between the true and approximate CDFs is controlled by the estimation error of the unclosed terms. Numerical experiments with analytically solvable test problems confirm that the sampling-based approximation converges satisfactorily with the number of samples.
\end{abstract}

\maketitle

\section{Introduction}

The study of statistical conservation laws is motivated by the inherently statistical nature of turbulence, where uncertain measurements lead to markedly different predictions \cite{MYturbulence}. The fluid velocity and pressure $(u_f,p_f): (t,x)\in\mathbb R_+ \times \mathbb R^d \mapsto \mathbb R^d\times \mathbb R$ are governed by the incompressible Navier-Stokes equations (NSE), and uncertainties are naturally modeled as random perturbations of the initial and/or boundary data. Statistical conservation laws refer to kinetic equations (conservation laws) for the joint probability density functions (PDFs) $f^{(K)} = f^{(K)}(t, x_1, v_1,\dots, x_K, v_K)$ of the velocity field $u_f$ at multiple spatial locations  $\{x_i\}_1^K$ ($K\in\mathbb N$) and time instants \cite{Oberlack2010,friedrich2012}. For $Q^i \subset \mathbb R^d$ Borel sets ($i=1,\dots,K$), we have
\begin{equation} \label{eq:def_fK}
  \text{Prob}(\{ u_f(t,x_i)\in Q^i \text{ for all } i \}) = \int_{\prod Q^i} f^{(K)}(t,x_1,v_1,\dots,x_K,v_K) \, dv_1\cdots dv_K. 
\end{equation}

These joint PDFs encode the complete statistical information of the turbulent velocity field. For example, the mean velocity $\langle u_f \rangle(t,x)$ and the two-point correlation $S_r(\ell;x) = \langle |(u_f(x+\ell n)-u_f(x))\cdot n|^r\rangle$ (with $n\in\mathbb R^d$ a unit vector and $\ell>0$) are moments of $f:=f^{(1)}(t,x,v)$ and $f^{(2)}(t,x,v_1,x+\ell n, v_2)$, respectively. By $\langle\cdot\rangle$ we denote the mean with respect to the given probability measure. The Reynolds-averaged NSE governing $\langle u_f \rangle$ are central in engineering practice \cite{Popeturbulence}. For near-wall shear flows, $\langle u_f \rangle$ obeys a logarithmic law \cite{Popeturbulence}. As for the structure function $S_r(\ell;x)$, it follows Kolmogorov’s scaling law $S_r(\ell;x) \sim \ell^{\zeta_r}$ with $\zeta_r=r/3$ in the inertial range for $d=3$ \cite{Kolmogorov1941}. This scaling holds for $r=2,3$, but shows significant deviations for larger $r$ \cite{Ishihara2009}.

Due to the limited spatio-temporal validity of these phenomenological theories, statistical conservation laws provide a promising unified framework for capturing velocity statistics across all scales \cite{friedrich2012}. Generally they read as $\partial_t f^{(K)}=L_Kf^{(K)}$, with $L_K$ a problem-specific operator. As with many statistical formulations of NSE, however, these equations are \emph{not} closed: $\partial_t f^{(K)} = C_K f^{(K)} + U_K$, where $C_K$ is known and $U_K$ the unclosed part, typically arising from pressure and viscous terms in NSE. This term can be expressed as $L_{K,K+1}(f^{(K)},f^{(K+1)})$, forming the infinite Lundgren-Monin-Novikov (LMN) hierarchy \cite{Lundgren,Monin1967,Novikov1968}, or equivalently through unknown conditional averages such as $\langle \Delta_{x_i} u_f | u_f = v_i \rangle$ \cite{friedrich2012}. Various analytical closures for $U_K$ have been proposed, including statistical independence closures \cite{TATSUMI_2011}, `Markov-process-in-scale' ansatz for velocity increments \cite{Friedrich1997,friedrich2017}, and DNS-data-driven closures \cite{Boffetta2002}. 

Statistical conservations laws also arise in other incompressible flow problems. The `PDF method' for turbulent combustion employs velocity-composition joint PDFs to model turbulence-reaction coupling in the incompressible NSEs with varying densities \cite{POPE1985119}. Passive scalar transport in random (incompressible) velocity fields provides another classical example, with long-time scalar PDF behavior analyzed in \cite{Sinai1989}.

By contrast, statistical conservation laws for compressible flows and systems of hyperbolic balance laws remain largely unexplored. When the divergence-free condition is absent, even the spatial transport term of the kinetic equations acquires nonlocal contributions (see Appendix~\ref{app:sys} for preliminary PDF equations for 1D systems). These PDF equations may connect to the theory of statistical solutions of conservation laws \cite{Fjordholm2017,Fjordholm2018,Lanthaler2021,giess2021}. A key difficulty is dimensionality: the $K$-point PDF in (\ref{eq:def_fK}) depends on $1+2dK$ variables. Therefore, most existing work focuses on the scalar setting:
\begin{equation} \label{eq:scalar}
\begin{split}
  & \partial_tu + \nabla_{x} \cdot g(u) = Q(u,\nabla_x u, \Delta_x u), \\
  & u(0,x,\xi(\omega)) = u_0(x,\xi(\omega)),
\end{split}
\quad t>0, x\in D \subset \mathbb R^d, \omega\in \Omega,
\end{equation}
where $u = u(t,x, \xi(\omega)) \in \mathbb R$, $g=g(u)$ is a smooth flux in $\mathbb R^d$, $\xi: \Omega\to \mathbb R^{N_\xi}$ is an absolutely-continuous random variable with a prescribed density function $p(\xi)$, and $Q$ represents possible source terms. For scalar conservation laws ($Q=0$), the `method of distributions' derives governing equations for single-point cumulative distribution functions (CDFs), which involve only spatial transport \cite{tartakovsky2015,boso2020}. If shocks occur in the solution $u$, a kinetic defect source term must be added \cite{perthame2002}, which is generally unknown a prior but may be inferred from data \cite{boso2020}. When $Q=Q(u)$, the PDF equation is closed for smooth $u$ \cite{Ven2012}, equivalent to the mentioned CDF formulation (see also Section \ref{sec:derivation}). For $Q=Q(u,\nabla_x u)$, closed equations for the joint $u$–$\nabla_x u$ PDF can be derived in one dimension for smooth solutions \cite{Ven2012}, while discontinuities will again introduce defect terms to these kinetic equations.

Another important line of research concerns the viscous Burgers equation with random forcing, corresponding to $g(u)=u^2/2$ and $Q=\epsilon \Delta_x u$ in (\ref{eq:scalar}), often viewed as a scalar analogue of NSE without pressure \cite{BEC20071,BENZI20231}. Governing equations for single- and two-point PDFs have been analyzed, some with analytical closure proposals \cite{Polyakov1995,Gotoh1998,E2000}. Notably, these studies clarified the statistics of velocity gradients and resolved a long-standing debate about their tail behavior \cite{E2000}.

The objective of this paper is to investigate numerical approximations to statistical conservation laws for hyperbolic balance laws, focusing on the scalar case (\ref{eq:scalar}). Unlike earlier studies \cite{Ven2012,cho2014,boso2020}, we consider viscous balance laws with $Q = \epsilon\Delta_x u$ in (\ref{eq:scalar}) to mimic the interplay between hyperbolic transport and viscosity operators. The viscous term ensures global existence of classical solutions of $u$, but leads once again to unclosed statistical conservation laws. In Section~\ref{sec:derivation} we systematically derive these laws. A previously-undocumented \emph{dissipative anomaly} is further revealed: for nonlinear fluxes $g(u)$ in (\ref{eq:scalar}), the viscosity-related unclosed term $U_K$ remains essential even in the vanishing viscosity limit, as it guarantees the positivity of PDFs. Since analytical closures of $U_K$ often introduce very different model structures, our approach does not attempt to derive them. Instead, we approximate $U_K$ directly from sampled solutions of $u$, a simple yet, to our knowledge, unexplored pipeline. This sampling-based method is flexible and extends naturally to statistical conservation laws of other systems, including NSE. To validate it, we construct analytical test cases for (\ref{eq:scalar}). Numerical results confirm the a priori error analysis and show that PDF/CDF errors diminish with increasing sample size.

Finally, we remark that the present setting (\ref{eq:scalar}) coincides with random hyperbolic conservation laws. Many numerical strategies exist for such problems, including Monte Carlo, stochastic collocation, and generalized polynomial chaos expansions (see, e.g., \cite{ABGRALL2017507,beck2020,chertock2024}). Recently, methods based on Young-measure-valued solutions and entropy minimization have been developed \cite{chu2025}. Most of these approaches, however, target only pointwise statistics (moments or single-point PDFs). In contrast, the statistical conservation law framework naturally extends to multi-point joint distributions, offering a structure-based route for capturing correlations as an essential feature in turbulence. Even for single-point PDFs, our numerical experiments indicate that leveraging the PDE structure (though unclosed) `enhances sample efficiency' and yields more accurate results than reconstructing the distribution directly from samples, as in Monte Carlo-type methods. Thus, statistical conservation laws provide a unifying and potentially more predictive perspective.

The rest of the paper is organized as follows. Section~\ref{sec:derivation} derives the statistical conservation laws for (\ref{eq:scalar}), with extensions to systems in Appendix A. Section~\ref{sec:pos} analyzes positivity-preserving properties of the PDF equations. Section~\ref{sec:approx} introduces the sampling-based approximation algorithm. Analytical and numerical test cases are discussed in Sections~\ref{sec:cases} and \ref{sec:numerical}, respectively. Conclusions are drawn in Section~\ref{sec:conclusion}.

\section{Statistical conservation laws for scalar problems (\ref{eq:scalar})} \label{sec:derivation}

In this section we derive the statistical conservation laws for the scalar balance law (\ref{eq:scalar}) with viscous term $Q=\epsilon \Delta_x u$ ($\epsilon>0$ constant) and random initial data. Both PDF- and CDF-based formulations are presented. Two different formal derivations are outlined.

Assume that the $K$-point joint PDF of $u$ exists and is defined by
\begin{equation} \label{eq:f_def}
  f^{(K)} = f^{(K)}(t,x_1,v_1,\dots,x_K,v_K): = \mathbb E_\xi \left[ \prod\nolimits_{i=1}^K \delta\Big(u(t,x_i,\xi(\omega))-v_i\Big) \right],
\end{equation}
with $v_i\in\mathbb R$ and $\delta(\cdot)$ denoting the Dirac distribution. The expectation $\mathbb E_\xi$ is taken with respect to the random variable $\xi(\omega)$ associated with the initial condition in (\ref{eq:scalar}). The PDF is nonnegative $f^{(K)}\ge 0$ and depends on $1+K(d+1)$ internal variables. It satisfies $\int_{\mathbb R^K} f^{(K)} \prod dv_i = 1$ (normalization), $\int_{\mathbb R} f^{(K)}(\cdot,x_K,v_K) \,dv_K = f^{(K-1)}$ (marginalization), and exhibits the `coincidence' behaviors like \cite{Lundgren}
\[
  \lim_{x_\beta\to x_\alpha} f^{(K)}(\cdot,x_\alpha,v_\alpha,x_\beta,v_\beta) = f^{(K-1)}(\cdot,x_\alpha,v_\alpha)\delta(v_\alpha-v_\beta), \quad \forall \alpha,\beta=1,\dots,K.
\]
The governing equation for $f^{(K)}$ is
\begin{equation} \label{eq:fN}
\left\{
\begin{aligned}
  &\partial_tf^{(K)} + \sum_{i=1}^K \nabla_{x_i} \cdot \partial_{v_i}\left( g'(v_i)\int_{-\infty}^{v_i} f^{(K)}(\cdot,\tilde v_i) \, d\tilde v_i \right) + \sum_{i=1}^K \partial_{v_i}\left( m_i f^{(K)} \right) = 0, \\
  &m_i = m_i(t,x_1,v_1,\dots,x_K,v_K):= \mathbb E_\xi \left[ \epsilon \Delta_x u(t,x_i,\xi)|u(t,x_j,\xi) = v_j \text{ for all } j \right],
\end{aligned}
\right.
\end{equation}
which has to be completed by proper initial/boundary data. It is linear and integro-differential in $f^{(K)}$, featuring a spatially nonlocal convective term, an unclosed drift term in $v_i$-space through $m_i$, but no explicit source terms. Since $m_i$ depends on the joint distribution of $(u_1,\dots,u_K,\Delta_{x_i} u)$ by the definition of the conditional expectation, it cannot be expressed solely in terms of $f^{(K)}$, leaving (\ref{eq:fN}) unclosed.

Now define the $K$-point joint CDF $F^{(K)} = F^{(K)}(t,x_1,v_1,\dots,x_K,v_K)$ of $u$ as
\[
  F^{(K)} = \text{Prob}(\{u(\cdot,x_i)\le v_i \text{ for all } i\}) = 
  \int_{\prod(-\infty,v_i]} f^{(K)}(t,x_1,\tilde v_1,\dots,x_K,\tilde v_K) \prod d\tilde v_i.
\]
Clearly, $f^{(K)}=\partial_{v_1\cdots v_K}^K F^{(K)}$, and $F^{(K)}$ is non-decreasing in each $v_i$. Integrating (\ref{eq:fN}) over $v_1,\dots,v_K$ yields the governing equation for $F^{(K)}$:
\begin{equation} \label{eq:FN_eq}
\begin{split}
  & \partial_tF^{(K)} + \sum_{i=1}^K g'(v_i) \cdot \nabla_{x_i} F^{(K)} 
  \\ 
  &+ \sum_{i=1}^K \int_{\prod_{1\le j\le K, \, j\ne i}(-\infty,v_j]} \left( m_i f^{(K)} \right)(\cdot,\tilde v_{i-1},v_i,\tilde v_{i+1}) \prod\nolimits_{1\le j\le K,\, j\ne i} d\tilde v_j = 0, 
\end{split}
\end{equation}
where $f^{(K)} = \partial^K_{\tilde v_1 \cdots \tilde v_{i-1} v_i \tilde v_{i+1}\cdots \tilde v_K} F^{(K)}$, and $m_i$ is as in (\ref{eq:fN}). Compared to (\ref{eq:fN}), the spatial term in (\ref{eq:FN_eq}) is local, but the unclosed terms are more involved for $K\ge 2$.

For the single-point case $K=1$, write $f=f(t,x,v)=f^{(1)}$ and $F=F(t,x,v)=F^{(1)}$. Then (\ref{eq:fN}) and (\ref{eq:FN_eq}) reduce to
\begin{align}
  &\left\{
  \begin{aligned}
    &\partial_t f + \nabla_x\cdot \partial_v\left( g'(v) \int_{-\infty}^v f(t,x,\tilde v)\,d\tilde v \right) + \partial_v \left( m(t,x,v) f \right) = 0, \\
    &m(t,x,v):=\mathbb E_\xi \left[ \epsilon \Delta_x u | u=v \right].
  \end{aligned}
  \right. \label{eq:f_eq}
  \\
  &\left\{
  \begin{aligned}
    &\partial_t F + g'(v)\cdot \nabla_x F + m(t,x,v) \partial_v F = 0, \\
    &F(0,x,v) = G(x,v) := \text{Prob}(\{u_0(x,\xi)\le v\}) = \int_{\{z|u_0(x,z)\le v\}} p_\xi(z)\,dz.
  \end{aligned}
  \right. \label{eq:F_eq}
\end{align}
Eq. (\ref{eq:F_eq}) is a linear transport equation, which greatly facilitates the subsequent analysis. In the sequel we always assume that the solution to (\ref{eq:F_eq}) exists uniquely.

\begin{remark}
  In analogy with the LMN hierarchy for the incompressible NSE, the statistical conservation laws for scalar problems, (\ref{eq:fN}) and (\ref{eq:FN_eq}), can also be formulated as an infinite hierarchy of evolution equations for $f^{(1)}, f^{(2)}, \dots$. These results are presented in a separate work \cite{huang_hyp_proc}.
\end{remark}

\begin{remark}
  While our main focus is on the random balance laws (\ref{eq:scalar}), the framework of statistical conservation laws also applies to stochastic PDEs of the form $du + \nabla_x \cdot g(u)\,dt = \epsilon \Delta_x u\,dt + \sigma(x)\,dW_t$, where $W_t$ is a Brownian motion with real-valued diffusion coefficient $\sigma = \sigma(x)$. In this setting, the single-point CDF $F(t,x,v)$ is governed by a Fokker-Planck equation:
  \begin{equation} \label{eq:F_eq_st}
    \partial_t F + g'(v)\cdot \nabla_x F + m(t,x,v)\partial_v F = \tfrac{1}{2}\sigma(x)^2 \partial_v^2 F.
  \end{equation}
  See a rigorous derivation of (\ref{eq:F_eq_st}) in \cite{E2000}. We will return to (\ref{eq:F_eq_st}) in Remark \ref{rem:pos_st}.
\end{remark}

\begin{remark}
  For completeness, we record the statistical conservation laws for (\ref{eq:scalar}) with a source term $Q$ independent of $\Delta_x u$. First, if $Q=Q(u)$ and $u$ admits a local smooth solution, then $f=f(t,x,v)$ and $F=F(t,x,v)$ satisfy \emph{closed} equations:
  \[
  \begin{aligned} 
    &\partial_t f + \nabla_x\cdot \partial_v\left( g'(v) \int_{-\infty}^v f(t,x,\tilde v) \, d\tilde v \right) + \partial_v(Q(v)f) = 0, \\
    &\partial_t F + g'(v)\cdot \nabla_x F + Q(v)\partial_v F = 0,
  \end{aligned}
  \]
  which have the same structures as (\ref{eq:f_eq}) and (\ref{eq:F_eq}) with $m(t,x,v)$ replaced by $Q(v)$.

  Next, if $Q=Q(u,\nabla_x u)$ and $u$ is smooth locally, then one obtains a closed equation for the joint PDF $\mathfrak p = \mathfrak p(t,x,v,w) := \mathbb E_\xi \left[ \delta(u(t,x,\xi)-v) \cdot \delta(\nabla_x u(t,x,\xi)-w) \right]$ of $u$ and its gradient ($v\in\mathbb R$, $w\in\mathbb R^d$). For notational simplicity, following \cite{Ven2012}, set $\mathcal N(u,\nabla_x u) := \nabla_x\cdot g(u) - Q(u,\nabla_x u)$ so that $\partial_t u + \mathcal N(u,\nabla_x u) = 0$. Then $\mathfrak p$ satisfies
  \[
  \left\{
  \begin{aligned}  
    &\partial_t \mathfrak p = \partial_v \left( \mathcal N(v,w) \mathfrak p \right) + \nabla_w \cdot \left( \diffp{\mathcal N}{u}(v,w)w \mathfrak p \right) + \nabla_w \cdot \left( \diffp{\mathcal N}{\nabla_x u}(v,w) \mathcal A(t,x,v,w) \right), \\
    &\nabla_w \cdot \mathcal A(t,x,v,w) = -\nabla_x \mathfrak p - w \partial_v \mathfrak p,
  \end{aligned}
  \right.
  \]
  with $\mathcal A$ representing $\mathbb E_\xi \left[ (\nabla_x\nabla_x u)\cdot\delta(u-v)\cdot\delta(\nabla_x u-w) \right] \in \mathbb R^{d\times d}$. While this system appears underdetermined for $d\ge 2$, it closes in one dimension. Indeed, for $d=1$, the second equation reduces to $\partial_w \mathcal A = -\partial_x \mathfrak p - w \partial_v \mathfrak p$, which gives $\mathcal A = - \int_{-\infty}^w (\partial_x \mathfrak p + \tilde w\partial_v \mathfrak p)(\cdot,v,\tilde w)\, d\tilde w$. Substituting this $\mathcal A$ into the first equation produces a closed system for $\mathfrak p$ in 1D \cite{Ven2012}.
\end{remark}

The remainder of this section is devoted to the formal derivation of (\ref{eq:fN}) and (\ref{eq:f_eq}). The approach of \cite{Ven2012} is outlined in Section \ref{subsec:formal}. 
An alternative, more intuitive derivation of (\ref{eq:F_eq}) is given in Section \ref{subsec:int}.

\subsection{Formal derivation of (\ref{eq:fN}) and (\ref{eq:f_eq})} \label{subsec:formal}

We first record two relations that will be repeatedly used (here $u_i:=u(t,x_i,\xi)$):
\begin{align}
  &\partial_\eta f^{(K)} = - \sum\nolimits_{k=1}^K \partial_{v_k} \mathbb E_\xi \left[ \partial_\eta u_k \prod\nolimits_{i=1}^K \delta(u_i-v_i) \right], \qquad\qquad\qquad \eta=t,x, \label{eq:chain} \\
  &\mathbb E_\xi \left[ u_k \prod\nolimits_{i=1}^K \delta(u_i-v_i) \right] = v_k \mathbb E_\xi \left[ \prod\nolimits_{i=1}^K \delta(u_i-v_i) \right] = v_k f^{(K)}, \quad k=1,\dots,K. \label{eq:replace}
\end{align}
For clarity we verify these relations for $K=1$; the extension to $K>1$ is immediate.

\textbf{Validation of (\ref{eq:chain}).} For smooth compactly supported $\varphi(v)$ (here $u=u(t,x,\xi)$),
\[
  \int \varphi(v) \mathbb E_\xi [\delta(u-v)] \, dv = \mathbb E_\xi \left[ \int \varphi(v) \delta(u-v) \, dv \right] = \mathbb E_\xi \left[ \varphi(u) \right].
\]
Differentiating in $\eta$ and using $\varphi'(u)\partial_\eta u = \int \varphi'(v)\delta(u-v)dv \cdot \partial_\eta u$, we obtain
\[
\begin{aligned}
  \int \varphi(v) &\partial_\eta \mathbb E_\xi [\delta(u-v)] \, dv = \mathbb E_\xi \left[ \varphi'(u)\partial_\eta u \right]
  = \mathbb E_\xi \left[ \int \varphi'(v)\delta(u-v) \, dv \cdot \partial_\eta u \right] \\
  &= \int \varphi'(v) \mathbb E_\xi \left[ \partial_\eta u \cdot \delta(u-v) \right] \, dv
  =-\int \varphi(v) \partial_v \mathbb E_\xi \left[ \partial_\eta u \cdot \delta(u-v) \right] \, dv.
\end{aligned}
\]
Since $\varphi$ is arbitrary, (\ref{eq:chain}) follows (with $f=\mathbb E_\xi[\delta(u-v)]$ by (\ref{eq:f_def})).

\textbf{Validation of (\ref{eq:replace}).} The desired identity is derived as follows.
\[
\begin{aligned}
  \int \varphi(v) \mathbb E_\xi \left[ u \delta(u-v) \right] \, dv &= \mathbb E_\xi \left[ \int \varphi(v) u \delta(u-v) \, dv \right]
  = \mathbb E_\xi \left[ \varphi(u)u \right] \\
  &= \mathbb E_\xi \left[ \int \varphi (v)v \delta(u-v) \, dv \right] = \int \varphi(v) v \mathbb E_\xi \left[ \delta(u-v) \right] \, dv.
\end{aligned}
\]

We now derive (\ref{eq:f_eq}); the extension to the multi-point case (\ref{eq:fN}) is straightforward. It is seen from (\ref{eq:chain}) and (\ref{eq:scalar}) that
\begin{equation} \label{eq:feq_mid}
\begin{split}
  \partial_t f &= -\partial_v \mathbb E_\xi \left[ \partial_t u \cdot \delta(u-v) \right] = \partial_v \mathbb E_\xi \left[ \nabla_x\cdot g(u) \delta(u-v) \right] - \partial_v \mathbb E_\xi \left[ \epsilon\Delta_x u \cdot \delta(u-v) \right] \\
  &= \partial_v \left( g'(v) \cdot \mathbb E_\xi \left[ \nabla_x u \cdot\delta(u-v) \right] \right) - \partial_v \mathbb E_\xi \left[ \epsilon\Delta_x u \cdot \delta(u-v) \right]
\end{split}
\end{equation}
and
\[
  \nabla_x f = -\partial_v \mathbb E_\xi \left[ \nabla_x u \cdot \delta(u-v) \right].
\]
The last expression may lead to
\begin{equation} \label{eq:close_trans}
  \mathbb E_\xi \left[ \nabla_x u \cdot \delta(u-v) \right] = -\int_{-\infty}^v \nabla_x f(\cdot,\tilde v) \, d\tilde v.
\end{equation}
Substituting (\ref{eq:close_trans}) into (\ref{eq:feq_mid}) gives the $f$-equation (\ref{eq:f_eq}), provided one uses the standard relation (see \cite{Popeturbulence,friedrich2012,vond2019})
\[
  \mathbb E_\xi \left[\epsilon \Delta_x u \cdot \delta(u-v)\right] = \mathbb E_\xi \left[\epsilon \Delta_x u | u=v \right] f(\cdot,v).
\]

\begin{remark} \label{rem:inviscid}
The above derivation is rigorous for viscous scalar balance laws with $\epsilon>0$. In the inviscid case $\epsilon=0$, (\ref{eq:feq_mid}) becomes $\partial_t f = \partial_v \left( g'(v) \cdot \mathbb E_\xi \left[ \nabla_x u \cdot \delta(u-v) \right] \right)$, which, after applying (\ref{eq:close_trans}), leads formally to conservation laws of the same form as (\ref{eq:f_eq}) and (\ref{eq:F_eq}) with the viscous term absent. These equations are valid as long as the solution $u$ remains smooth, but fail to hold once shocks form (see more discussions in Section~\ref{sec:pos}). The reason is that after shock formation, the quantity $\mathbb E_\xi[\nabla_x u ,\delta(u-v)]$ may become discontinuous in $v$, rendering (\ref{eq:close_trans}) invalid. This breakdown will be illustrated by Example~\ref{ex:nonpos}.
\end{remark}

\subsection{An intuitive derivation of (\ref{eq:F_eq})} \label{subsec:int}

We now present an intuitive derivation of (\ref{eq:F_eq}). Recall that $p(\xi)$ is the PDF of $\xi$. Start from the definition of the single-point CDF:
\[
  F(t,x,v) = \text{Prob}(\{u(t,x,\xi)\le v\}) = \int_{\{\xi: \ u(t,x,\xi)\le v \}} p(\xi)\,d\xi.
\]
Consider the map $\Xi_{t,x}: \xi \mapsto u(t,x,\xi) = v$ parameterized by $(t,x)$, and let $\Xi_{t,x}^{-1}(v)$ denote the set of preimages of $v$. Then $F(v)-F(v-) = \int_{\Xi_{t,x}^{-1}(v)} p(\xi)\, d\xi$. Since $F$ is non-decreasing, it is continuous in $v$ for almost every $v$. Consequently, for a.e. $v$, either $m(\Xi_{t,x}^{-1}(v))=0$ (here $m$ is the Lebesgue measure), or $p(\xi)=0$ for a.e. $\xi \in \Xi_{t,x}^{-1}(v)$. Without loss of generality, we assume $m(\Xi_{t,x}^{-1}(v))=0$ and $p(\xi) > 0$ on $\Xi_{t,x}^{-1}(v)$. It then follows that
\[
  f(v) = \partial_v F = \sum_{\xi\in \Xi_{t,x}^{-1}(v)} p(\xi) |\partial_\xi u|^{-1}.
\]
If $u(\xi) \in C^1$, then for each $\xi = \xi(t,x,v) \in \Xi_{t,x}^{-1}(v)$, the implicit function theorem gives $(\partial_\xi u) (\partial_t\xi) = -\partial_t u$, $(\partial_\xi u) (\nabla_x\xi) = -\nabla_x u$, and $(\partial_\xi u) (\partial_u\xi) = 1$. Hence, for a.e. $v$,
\[
\begin{aligned}
  \partial_t F &= \sum_{\xi\in\Xi_{t,x}^{-1}(v)} p(\xi) (\partial_t \xi) \cdot \text{sgn}(\partial_\xi u) = -\sum_{\xi\in\Xi_{t,x}^{-1}(v)} (\partial_tu) p(\xi) |\partial_\xi u|^{-1}, \\
  \nabla_x F &= \sum_{\xi\in\Xi_{t,x}^{-1}(v)} p(\xi) (\nabla_x \xi) \cdot \text{sgn}(\partial_\xi u) = -\sum_{\xi\in\Xi_{t,x}^{-1}(v)} (\nabla_x u) p(\xi) |\partial_\xi u|^{-1}.
\end{aligned}
\]
Applying the governing equation (\ref{eq:scalar}) for $u$, we obtain
\begin{equation} \label{eq:F_eq_int}
  -(\partial_t F + g'(u)\cdot \nabla_xF) = \epsilon \sum_{\xi\in\Xi_{t,x}^{-1}(v)} (\Delta_x u) p(\xi) |\partial_\xi u|^{-1}.
\end{equation}
To handle the right-hand side, consider the joint PDF of $\Delta_x u$ and $u$ (for $a,v\in\mathbb R$): $\mathfrak p(a,v) := \mathbb E_\xi \left[ \delta(\Delta_x u(t,x,\xi)-a) \cdot \delta(u(t,x,\xi)-v) \right] = \int \delta(\Delta_xu(t,x,\xi)-a) \cdot \delta(u(t,x,\xi)-v) p(\xi)\,d\xi$.
By the definition of conditional expectation and formal properties of the Dirac delta distribution, we have
\[
\begin{aligned}
  \epsilon f(\cdot,v) \mathbb E_\xi[\Delta_xu|u=v] &= \epsilon \int a \mathfrak p(a,v)\,da
  = \epsilon \int \Delta_xu(t,x,\xi)\cdot\delta(u(t,x,\xi)-v)p(\xi)\,d\xi \\
  &= \epsilon \sum_{\xi\in\Xi_{t,x}^{-1}(v)} \Delta_xu(t,x,\xi)p(\xi)|\partial_\xi u|^{-1}.
\end{aligned}
\]
Substituting this into the right-hand side of (\ref{eq:F_eq_int}) immediately recovers (\ref{eq:F_eq}).

\section{Positivity of PDFs} \label{sec:pos}

This section addresses an important observation regarding the unclosed term $m=m(t,x,v)=\epsilon \mathbb E_\xi \left[ \Delta_x u | u=v \right]$ in the statistical conservation laws (\ref{eq:f_eq}) and (\ref{eq:F_eq}). Although $m$ contains the factor $\epsilon$, we show that, for nonlinear fluxes $g(v)$,
\[ m(t,x,v) \nrightarrow 0 \quad \text{as} \quad \epsilon\to 0. \]
This non-vanishing behavior can be viewed as a scalar analogue of the so-called dissipative anomaly in turbulence, while the underlying mechanism here is closely related to the shocks formed in the inviscid conservation laws. The same phenomenon occurs for the multipoint terms $m_i$ in (\ref{eq:fN}) and (\ref{eq:FN_eq}): in general, at least one $m_i$ does not vanish in the inviscid limit $\epsilon \to 0$. The reason is that these terms are essential for preserving the positivity of PDFs, or equivalently, the monotonicity of CDFs $F^{(K)}(\cdot,v_i)$ with respect to $v_i$ ($\forall i=1,\dots,K$). We illustrate this `by contradiction': the positivity of $f(v)$ is not preserved \emph{globally} and \emph{unconditionally} by the following equation with the unclosed term dropped:
\begin{equation} \label{eq:F_eq_noeps}
\begin{split}
  \partial_t F_{S0} + g'(v)\cdot \nabla_x F_{S0} &= 0, \\ 
  F_{S0}(0,x,v) &= G(x,v),
\end{split}
\end{equation}
where $G(x,v)$ has sufficient regularity and is non-decreasing in $v$ for all $(x,v)$. Then we have

\begin{theorem} \label{thm:nonpos}
  The solution $F_{S0}(t,x,v)$ to (\ref{eq:F_eq_noeps}) is non-decreasing in $v$ for all $(t,x,v)$ if and only if $g''(v)\cdot \nabla_x G \le 0$ for all $(x,v)$.
\end{theorem}
\begin{proof}
  By the method of characteristics, the solution is $F_{S0}(t,x,v) = G(x-g'(v)t, v)$, which gives
  \[
    \partial_v F_{S0}(t,x,v) = -tg''(v) \cdot \nabla_x G(x-g'(v)t,v) + \partial_vG(x-g'(v)t,v).
  \]
  Since $\partial_v G \ge 0$, the monotonicity $\partial_v F_{S0} \ge 0$ holds if $g''(v)\cdot \nabla_x G \le 0$ for all $(x,v)$. Conversely, if $g''(v)\cdot \nabla_x G(x_0,v_0) > 0$ for some $(x_0,v_0)$, then for sufficiently large $t>\partial_vG(x_0,v_0)/(g''(v_0)\cdot\nabla_x G(x_0,v_0))$, one can find $x=x_0+g'(v_0)t$ such that $\partial_v F_{S0}(t,x,v_0)<0$.
\end{proof}

\begin{corollary} \label{cor:nonpos}
  For $d=1$ and a convex flux $g(v)$ with $g''(v) \neq 0$ for all $v$, $F_{S0}(t,x,v)$ is non-decreasing in $v$ for all $(t,x,v)$ if and only if
  \[\left\{
  \begin{aligned}
    &\partial_x G \le 0 \text{ for all } (x,v), \quad \text{if } g''(v)>0, \\
    &\partial_x G \ge 0 \text{ for all } (x,v), \quad \text{if } g''(v)<0.
  \end{aligned}
   \right.\]
\end{corollary}

Theorem~\ref{thm:nonpos} and Corollary~\ref{cor:nonpos} imply that, for nonlinear $g(v)$, there always exist physically admissible initial CDFs $G(x,v)$ such that the solution to (\ref{eq:F_eq_noeps}) loses monotonicity in $v$, rendering it non-physical. Hence, (\ref{eq:F_eq_noeps}) does not constitute the correct statistical conservation law in the inviscid limit. We give a 1D example to illustrate this point.

\begin{example} \label{ex:nonpos}
  Consider the 1D scalar Burgers equation $\partial_t u + \partial_x(\frac{u^2}{2}) = 0$ on $D=[0,2\pi]$ with periodic boundary conditions and initial data $u_0(x,\xi) = \sin x + \xi$, where $\xi$ is a zero-mean Gaussian variable with standard deviation $\sigma$. The initial CDF is
  \begin{equation} \label{eq:G1}
    G(x, v) = \Phi\left(\frac{v-\sin x}{\sigma}\right),
  \end{equation}
  where $\Phi(x) = \frac{1}{\sqrt{2\pi}}\int^x \exp(-\frac{s^2}{2})\,ds = \frac{1}{2}\text{erfc}(-\frac{x}{\sqrt{2}})$ is the standard normal CDF. Here, $\partial_x G$ can be positive, generating negative $\partial_v F_{S0}$ according to Corollary~\ref{cor:nonpos}. Indeed,
  \begin{subequations} \label{eq:FS0_case1}
  \begin{align}
  &F_{S0}(t,x,v) = G(x-vt,v) = \Phi\left( \frac{v-\sin(x-vt)}{\sigma} \right), \\
  &f_{S0}(t,x,v):=\partial_v F_{S0} = \frac{1}{\sqrt{2\pi}\sigma}\exp \left( -\frac{(v-\sin(x-vt))^2}{2\sigma^2} \right) \left( 1+t\cos(x-vt) \right).
  \end{align}
  \end{subequations}
  Given any $x$, $f_{S0}$ starts to take negative values at $t=1$, coinciding with shock formation in the underlying Burgers equation.

  We continue the discussion from Remark~\ref{rem:inviscid} to show that, in this example, $\mathbb E_\xi \left[ (\partial_x u) \delta(u-v) \right]$ becomes discontinuous in $v$ for $t>1$, thereby invalidating (\ref{eq:close_trans}). 
  For $t>1$, the (unique entropic weak) solution $u(t,x,0)$ develops a shock at $x=\pi$, with $u(t,\pi^-,0) = v_s > 0$, $u(t,\pi^+,0) = -v_s$, and the product $|\partial_xu(t,\pi^-,0)\,\partial_x u(t,\pi^+,0)|$ is bounded away from zero. For other values of $\xi$, the shock in $u(t,x,\xi)$ is shifted to the position $\pi+\xi t \pmod{2\pi}$ according to (\ref{eq:uburgers}). Moreover, we will show in Theorem~\ref{thm:monotonicity} that the (classical) solution $u(t,\pi,\xi)$ is strictly increasing in $\xi$ for any $\epsilon>0$. Hence, in the inviscid limit $\epsilon\to 0$, the monotonicity persists for the entropic weak solution, implying $u(t,\pi,0^-) = -v_s$ and $u(t,\pi,0^+) = v_s$. Consequently,
  \[
    \mathbb E_\xi\!\left[ (\partial_x u) \delta(u-v) \right](t,\pi,v) = \int \partial_x u(t,\pi,z)\,\delta(u(t,\pi,z)-v)\,p(z)\,dz
  \]
  vanishes for $-v_s < v < v_s$ but remains bounded away from zero as $v \nearrow -v_s$ or $v \searrow v_s$. Thus, this quantity is discontinuous at $(t>1, x=\pi, v=\pm v_s)$ in the inviscid limit, which renders (\ref{eq:close_trans}) invalid. Intuitively, a `correction term' must be incorporated into (\ref{eq:close_trans}) in this limit. This correction is characterized by the nonzero contribution of the $m$-containing term in the statistical conservation laws.

  This example will be further analyzed in Section~\ref{subsec:case_1} and verified numerically in Section~\ref{subsec:test_1} (referred to as Case~\textbf{I}).
\end{example}

\begin{remark} \label{rem:pos_st}
  For stochastic scalar balance laws, the `zero-$\epsilon$' limit of (\ref{eq:F_eq_st}), $\partial_t F + g'(v)\cdot \nabla_x F = \tfrac{1}{2}\sigma^2 \partial_v^2 F$ with $F(0,t,v) = G(x,v)$, does not globally and unconditionally preserve the positivity of $f(v)$. Indeed, differentiating with respect to $v$ gives the governing equation for $f=\partial_v F$:
  \[
    \partial_t f + g'(v)\cdot \nabla_x f + g''(v) \cdot \nabla_x F = \tfrac{1}{2}\sigma^2 \partial_v^2 f, \quad 
    f(0,x,v) = \partial_v G(x,v).
  \]
  Take an initial condition $\partial_v G(x,v)$ such that $\partial_v G(x_0,v_0)=0$ at a local minimum, and assume $g''(v)\cdot\nabla_x G > \frac{1}{2}\sigma^2 \partial_v^2(\partial_v G)$ at $(x_0,v_0)$. Clearly, such a $G$ can be constructed. Then, for $t>0$, the solution $f(t,x,v)$ becomes negative in a neighborhood of $(x_0,v_0)$. We note that, for $g(v)=v^2/2$, an exact statistical conservation law in the inviscid limit was obtained in \cite{E2000}.
\end{remark}

The above discussion highlights that the unclosed term $m(t,x,v)$ in the statistical conservation laws plays a crucial role: it automatically ensures the positivity of $f(\cdot,v)$ for all $(t,x,v)$. In practice, however, $m(t,x,v)$ is unknown and can only be approximated via closure models or numerical sampling. It would therefore be desirable to develop a `monotonicity-preserving' criterion for the reconstructed $m(t,x,v)$ in (\ref{eq:F_eq}). Generally, the available condition is highly implicit, derived through the method of characteristics, as presented below. This analysis provides insights into several special cases, which we summarize as follows.

\begin{proposition} \label{prop:special}
  Consider the linear transport equation (\ref{eq:F_eq}) with unknown $m(t,x,v)$ and a physically admissible initial CDF $G(x,v)$. Suppose that $m(t,x,v)$ is sufficiently regular to ensure well-posedness of the characteristic system associated with (\ref{eq:F_eq}). Then:

  (i). There exist infinitely many choices of $m(t,x,v)$ such that $F(\cdot,v)$ remains a CDF for all $(t,x)$.

  (ii). In the spatially homogeneous case (that (\ref{eq:F_eq}) is independent of $x$), if $m(t,v)$ satisfies $|m(t,v)| \le C(t)(1+|v|)$ and $C(t)\in L^1([0,\infty))$, then $F(\cdot,v)$ remains a CDF for all $t>0$.
  
  (iii). If $m=m(t)$ (no dependence on $x$ or $v$) and $g''>0$ (convex laws), then there exist initial CDFs $G(x,v)$ such that $F(\cdot,v)$ becomes decreasing in $v$. This indicates that $m$ must depend on $x$ and/or $v$ in general. 
\end{proposition}

The first statement is immediate: any $\epsilon>0$ in (\ref{eq:scalar}), together with a compatible initial datum consistent with $G(x,v)$, produces one valid $m(t,x,v)$. To analyze Statements (ii) and (iii), we examine the characteristic system of (\ref{eq:F_eq}) for $\psi(\tau;t,x,v) := (X(\tau;t,x,v), V(\tau;t,x,v))$, which satisfies
\begin{equation} \label{eq:charc_de}
  \frac{d}{d\tau} \psi = \left( g'(V(\tau)), \, m(\tau,\psi(\tau)) \right), \quad \psi|_{\tau=t} = (x, \, v).
\end{equation}
Thus,
\begin{equation} \label{eq:F_char_sol}
  F(t,x,v) = G(\psi(0;t,x,v)).
\end{equation}

The requirement $f(\cdot,v)\ge 0$ can be expressed as (omitting $(t,x,v)$-dependence):
\begin{equation} \label{eq:realizability}
  \partial_v F = \nabla_{(x,v)} G(\psi(0)) \cdot \partial_v \psi(0) \ge 0,
\end{equation}
for all $(t,x,v)$. Here, $\partial_v \psi(\tau)$ evolves according to
\[
  \frac{d}{d\tau}\partial_v \psi = \left( g''(V(\tau)) \partial_vV, \, \nabla_{(x,v)} m(\tau,\psi(\tau))\cdot \partial_v \psi(\tau) \right), \quad
  (\partial_v\psi)|_{\tau=t} = (\bm{0}, \, 1),
\]
with $\bm{0}\in\mathbb R^d$. This yields
\begin{equation} \label{eq:charc_v}
  \begin{bmatrix} \partial_vX(0) \\ \partial_vV(0) \end{bmatrix} = e^{-\int_0^t A(\tau) d\tau} \begin{bmatrix} \bm{0} \\ 1 \end{bmatrix} \text{ with } A(\tau) = \begin{bmatrix} \bm{0}_{d\times d} & g''(V(\tau)) \\ \nabla_x m(\tau,\psi(\tau))^T & \partial_vm(\tau,\psi(\tau)) \end{bmatrix}.
\end{equation}
Therefore, condition (\ref{eq:realizability}), together with (\ref{eq:charc_de}) and (\ref{eq:charc_v}), provides a highly implicit realizability criterion for $m(t,x,v)$ ensuring the positivity of $f(\cdot,v)$.

\begin{proof}[Proof of Proposition~\ref{prop:special}]
(ii). For the spatially homogeneous case with $F=F(t,v)$, condition~(\ref{eq:realizability}) reduces to
\[
  \partial_vF(t,v) = G'(V(0))\partial_vV(0) \ge 0.
\]
From (\ref{eq:charc_v}), we have
\[
  \partial_vV(0;t,v) = \exp\left(-\int_0^t \partial_vm(\tau,V(\tau;t,v))\,d\tau \right) > 0,
\]
which shows that the nonnegativity of $f(t,v)$ is preserved for all $(t,v)$.  
Since $F(t,v)=G(V(0;t,v))$, it remains to verify that $V(0;t,v)\to\pm\infty$ as $v\to\pm\infty$. Introduce $\tilde V(\tau) := V(t-\tau)$, so that $\tilde V(0)=v$, $\tilde V(t)=V(0)$, and
\[
  \frac{d}{d\tau}\tilde V(\tau) = -m(t-\tau,\tilde V(\tau)).
\]
Assume $v\gg 1$. Denoting by $M$ an upper bound of $\int_{\mathbb R_+}C(s)\,ds$, we obtain $\frac{d}{d\tau}\tilde V \ge -C(t-\tau)(1+\tilde V)$, which implies
\[
  \tilde V(t) \ge (v+1) \exp\left( -\int_0^t C(t-s)\,ds \right) - 1 \ge (v+1)e^{-M}-1.
\]
Thus $v\to\infty$ leads to $V(0)=\tilde V(t)\to\infty$ for any $t>0$.  
Similarly, for $v\ll -1$ we obtain $d\tilde V/d\tau \le C(t-\tau)(1-\tilde V)$ and hence $\tilde V(t) \le (v-1) e^{-M}+1$, which shows $V(0)=\tilde V(t)\to -\infty$ as $v\to -\infty$.

(iii). Consider the one-dimensional case ($x\in\mathbb R$). When $m=m(t)$, (\ref{eq:charc_v}) yields $\partial_v X(0) = -\int_0^t g''(V(\tau)) \, d\tau$, and $\partial_vV(0)=1$. Hence condition (\ref{eq:realizability}) becomes
\begin{equation} \label{eq:realizability_Ct}
  \partial_xG(\psi(0)) \int_0^t g''( V(\tau)) \, d\tau \le \partial_vG(\psi(0)).
\end{equation}
If $g''(v)>0$ for all $v$, condition (\ref{eq:realizability_Ct}) can be violated by choosing $\partial_v G(\psi(0))=0$ and $\partial_x G(\psi(0))>0$. Such $G(x,v)$ are clearly attainable as CDFs in $v$ for each $x$.
\end{proof}

\begin{remark}
  Regarding the special case $m=m(t)$, consider the CDF equation (\ref{eq:F_eq}) with $d=1$ (one-dimensional), $m=m(t)$, and a convex flux $g(v)$ with $g''(v)>0$ for all $v$. It is seen from (\ref{eq:realizability_Ct}) that the solution $F(t,x,v)$ preserves the monotonicity in $v$ for all $(t,x,v)$ if the initial condition $G(x,v)$ is non-increasing in $x$ for any $(x,v)$. This condition is also necessary if $\inf_v g''>0$ or $m\in L^1([0,\infty))$.
\end{remark}

\section{Numerical approximation of the unclosed term} \label{sec:approx}

We now proceed to approximate the unclosed term $m(t,x,v) = \epsilon\mathbb E_\xi[\Delta_xu|u]$ in (\ref{eq:F_eq}) by a data-driven estimator $\hat m(t,x,v)$ constructed from numerical samples. This straightforward procedure provides a unified framework for handling different types of statistical conservation laws (namely, to evaluate the PDFs/CDFs).

For (\ref{eq:F_eq}), the procedures are collected as in Algorithm~\ref{alg:scle}. The key step is to construct the sampling-based estimator $\hat m_N(t,x,v)$ in (\ref{eq:estimator}) for $m(t,x,v)$. Here, $K_{h_v}$ is a kernel with bandwidth $h_v>0$. A simple choice is the uniform kernel
\begin{equation} \label{eq:kerconst} 
    K_{h_v}(x) = \frac{1}{h_v}\mathbf 1_{|x|\le\frac{h_v}{2}}. 
\end{equation}
This means averaging $(\Delta_x u)_i$ over those samples $u_i$ that fall within distance $\tfrac{h_v}{2}$ of $v$. 

\begin{remark}
  In many cases, the estimator $\hat m_N$ can only be constructed on a spatio-temporal grid $\{(t^n,x_{\mathbf j})\}$ because the samples $u_i(t,x)$ and $(\Delta_x u)_i(t,x)$ are computed numerically. One has to interpolate $\hat m_N$ to off-grid points $(t,x,v)$ when solving (\ref{eq:charc_de_p}). Nevertheless, in this work, we avoid this issue by focusing on analytically solvable special cases (see Section~\ref{sec:cases}), so that $\hat m$ can be evaluated by (\ref{eq:estimator}) for all $(t,x,v)$.
\end{remark}

\begin{algorithm}[htbp]
\caption{Sampling-based approximation to the statistical conservation law (\ref{eq:F_eq})}
\label{alg:scle}
\begin{algorithmic}[1]
\STATE \textbf{Parameters:} Viscosity $\epsilon>0$ in (\ref{eq:scalar}), number of samples $N$, bandwidth $h_v>0$ for the estimator, and time step $\Delta t$.
\STATE \textbf{Input:} Initial condition $u_0(x,\xi)$, and PDF $p(\xi)$ of the random variable $\xi$ in (\ref{eq:scalar}), which thus determine the initial CDF $G(x,v)$ of (\ref{eq:F_eq}).
\STATE \textbf{Sampling:} Draw $N$ i.i.d.\ samples $\{\xi_i\}_{i=1}^N$ from $p(\xi)$.
\FOR{$i=1,\dots,N$}
  \STATE Solve the scalar balance law \eqref{eq:scalar} with initial data $u_0(x,\xi_i)$ to obtain
  \[
    u_i(t,x) := u(t,x,\xi_i), \quad (\Delta_x u)_i(t,x) := \Delta_x u(t,x,\xi_i).
  \]
\ENDFOR
\STATE \textbf{Estimator construction:} For each $(t,x,v)$, compute the Nadaraya-Watson-type estimator for the conditional expectation:
\begin{equation} \label{eq:estimator}
  \hat m_N(t,x,v) = 
  \begin{cases}
    \dfrac{\epsilon \sum_{i=1}^N (\Delta_x u)_i K_{h_v}(u_i-v)}
          {\sum_{i=1}^N K_{h_v}(u_i-v)}, & \text{if denominator $\ne 0$,}\\[1.2em]
    0, & \text{otherwise}.
  \end{cases}
\end{equation}
\STATE \textbf{Approximation of CDF:} Solve the approximating transport equation 
\begin{equation} \label{eq:Fappr}
    \hat F_t + g'(v)\cdot \nabla_x \hat F + \hat m_N(t,x,v)\hat F_v = 0, \quad
    \hat F(0,x,v) = G(x,v),
\end{equation}
via backward characteristic tracing for $\hat\psi(\tau; t,x,v) = (\hat X(\tau;t,x,v), \hat V(\tau;t,x,v))$. For each $(t,x,v)$, $\hat \psi$ satisfies
\begin{equation} \label{eq:charc_de_p}
  \frac{d}{d\tau} \hat\psi = \left( g'(\hat V(\tau)), \, \hat m_N(\tau,\hat \psi(\tau) \right), \quad \hat \psi|_{\tau=t} = (x, \, v).
\end{equation}
Use the first-order Euler method to compute $\hat\psi(0)$ with the time step $\Delta t$.
\STATE \textbf{Output:} Approximate CDF 
\begin{equation} \label{eq:Fhat_char_sol}
  \hat F(t,x,v)=G(\hat\psi(0)).
\end{equation}
\end{algorithmic}
\end{algorithm}

It is seen that, if $h_v$ is taken as a fixed value (i.e., independent of $N$), then by the law of large numbers, as $N \to \infty$, the estimator $\hat m_N(t,x,v)/\epsilon$ in (\ref{eq:estimator}) converges to
\[
  \frac{\int_{v-h_v/2}^{v+h_v/2} \int_{\mathbb R} a \mathfrak p_{\Delta_xu,u}(t,x,a,\tilde v)\,da\,d\tilde v}{\int_{v-h_v/2}^{v+h_v/2} f(t,x,\tilde v)\,d\tilde v} = \int_{v-h_v/2}^{v+h_v/2} \omega(t,x,\tilde v) \mathbb E_\xi[\Delta_xu|u](t,x,\tilde v)\,d\tilde v,
\]
where $\mathfrak p_{\Delta_xu,u}(t,x,a,v)$ denotes the joint PDF of $\Delta_xu(t,x) = a$ and $u(t,x)=v$, and $\omega(t,x,\tilde v) = f(t,x,\tilde v)/\int_{v-h_v/2}^{v+h_v/2} f(t,x,\tilde v)\,d\tilde v$. Clearly, this is the weighted average of $\mathbb E_\xi[\Delta_x u|u=\tilde v]$ over the interval $\tilde v\in[v-h_v/2,v+h_v/2]$ and in general it does not coincide with the desired quantity $m(v)/\epsilon=\mathbb E_\xi[\Delta_x u|u=v]$ whenever $h_v>0$ (it is therefore referred to as a `\emph{biased} estimator').

Therefore, in the limit $N \to \infty$, it is necessary that $h_v \to 0$ while simultaneously $Nh_v \to \infty$ (ensuring that the mean number of samples in the interval goes to infinity). It has been established that, for kernels such as (\ref{eq:kerconst}), if the `optimal' bandwidth is chosen as $h_v = C N^{-1/5}$, the mean integrated squared error (MISE) satisfies
\begin{equation} \label{eq:mise}
  \mathbb E_\xi\left[ \|\hat m_N - m\|^2_{L^2_v([v_{min},v_{max}])} \right] = C'N^{-4/5}  
\end{equation}
in the large $N$ limit \cite{Bierens_1994}. It should be emphasized that the coefficients $C, \, C'$ are usually \emph{not} known in prior as they are related to properties of $m(v)$.

We close this section with an analysis of the discrepancy between the approximated system (\ref{eq:Fappr}) and the original CDF equation (\ref{eq:F_eq}). In certain settings, the difference $\|F-\hat F\|$ caused by $\|m-\hat m\|$ can be quantified with an error estimate.

\begin{theorem} \label{thm:error}
  Assume that $g''$ is bounded with $|g''|\le C$. Consider both (\ref{eq:F_eq}) and (\ref{eq:Fappr}) on $(t,x,v)\in [0,T]\times D_{x,v} \subset \mathbb R_+\times \mathbb R\times \mathbb R$. Further assume that $\nabla_{x,v}(m,G)\in L^\infty([0,T]\times D_{x,v})$. Then we have
  \[
    \|(F-\hat F)(t,\cdot)\|_{L^\infty_{x,v}} \le e^{(C+\|\nabla_{x,v}m\|_{L^\infty})t} \|\nabla_{x,v}G\|_{L^\infty} \|m-\hat m_N\|_{L^\infty} t, \quad \forall \, t\le T.
  \]
\end{theorem}

\begin{proof}
  For the original CDF equation (\ref{eq:F_eq}), the method of characteristics (\ref{eq:charc_de}) yields the representation (\ref{eq:F_char_sol}): $F(t,x,v) = G(X(0),V(0))$. Similarly, for the perturbed system (\ref{eq:Fappr}), the solution $\hat F$ is given by (\ref{eq:charc_de_p}) and (\ref{eq:Fhat_char_sol}). Subtracting (\ref{eq:charc_de}) and (\ref{eq:charc_de_p}) and integrating over $\tau$, we obtain
  \[
  \begin{aligned}
    |&\psi - \hat \psi|(s) \le \int_s^t \left( |g'(V(\tau))-g'(\hat V(\tau))| + |m(\tau,\psi(\tau)) - \hat m_N(\tau,\hat \psi(\tau))| \right) \, d\tau \\
    & \le \int_s^t \left( C |V-\hat V|(\tau) + |m(\tau,\psi(\tau)) - m(\tau,\hat \psi(\tau))| + |m(\tau,\hat \psi(\tau)) - \hat m_N(\tau,\hat \psi(\tau))| \right) \, d\tau \\
    & \le \int_s^t \left( (C+\|\nabla_{x,v}m\|_{L^\infty}) |\psi-\hat \psi|(\tau) + \|m-\hat m_N\|_{L^\infty} \right)\,d\tau.
  \end{aligned}
  \]
  By Gronwall's inequality, it follows that
  \[
    |\psi - \hat \psi|(s) \le e^{(C+\|\nabla_{x,v}m\|_{L^\infty})(t-s)}\|m-\hat m_N\|_{L^\infty} \cdot (t-s), \quad s\in[0,t].
  \]
  Therefore, the desired error estimate is a direct consequence of
  \[
    \|(F-\hat F)(t,\cdot)\|_{L^\infty_{x,v}}\le \|\nabla_{x,v}G\|_{L^\infty} |\psi - \hat \psi|(0),
  \]
  which completes the proof.
\end{proof}

Theorem \ref{thm:error} shows that, under suitable assumptions, the error between the exact solution $F$ and the perturbed solution $\hat F$ is controlled by the discrepancy between $m(t,x,v)$ and $\hat m_N(t,x,v)$. Moreover, the bound indicates that the error can grow substantially in $t$. The conditions of the theorem (namely, boundedness of $g''$ and of $\nabla_{x,v}(m,G)$) are satisfied by the test cases presented in the subsequent sections.

\section{Parametric initial data} \label{sec:cases}

This section examines special cases of the scalar problem (\ref{eq:scalar}) where explicit analytical results can be obtained. These results will also serve as a foundation for the numerical tests presented in the next section. We focus on two classes of initial conditions for which the mapping $\Xi_{t,x}: \xi \mapsto u(t,x,\xi)=v$ is injective for every $(t,x)$. In such cases, we can define the inverse map $\xi=\Xi_{t,x}^{-1}(v)$, and the system is fully characterized by
\begin{align}
  F(t,x,v) &= \mathbb P(u(t,x,\xi)\le v) = \int_{-\infty}^{\Xi_{t,x}^{-1}(v)} p(\xi') \, d\xi', \label{eq:F_anal} \\
  f(t,x,v) &= \partial_v F(t,x,v) = p\circ \Xi_{t,x}^{-1}(v) \cdot \left[ \left. \partial_\xi u \right|_{\xi= \Xi_{t,x}^{-1}(v)} \right]^{-1}, \label{eq:f_anal}
\end{align}
while the unclosed term in (\ref{eq:f_eq}) and (\ref{eq:F_eq}) reduces to
\begin{equation} \label{eq:unc_anal}
  m(t,x,v)=\mathbb E_\xi[\Delta_x u(t,x,\xi)|u=v] = \left.\Delta_x u \right|_{\xi=\Xi_{t,x}^{-1}(v)},
\end{equation}
since the conditioning uniquely determines $\xi$. Consequently, in this case $m$ does not depend on the distribution $p(\xi)$ of the random variable $\xi$.

\subsection{First type of random initial data} \label{subsec:case_1}

Let $U_T = (0,T]\times \mathbb T^d$.
\begin{theorem} \label{thm:monotonicity}
  Consider the scalar problem (\ref{eq:scalar}) on $(t,x,\xi)\in U_T \times \mathbb R$ with initial data $u_0(x,\xi)$. Assume that $\partial_\xi u(\cdot,\xi)\in C^{1,2}(U_T)\cap C(\bar U_T)$ for every $\xi$.  
  If $\partial_\xi u_0 \ge 0$ (resp. $\partial_\xi u_0\le 0$) for all $(x,\xi)$ and $\partial_\xi u_0 \not\equiv 0$, then the solution $u(t,x,\xi)$ is strictly increasing (resp. decreasing) in $\xi$ for all $(t,x,\xi)\in U_T \times \mathbb R$.

  In particular, if $\lambda=\inf_{\mathbb T^d}\partial_\xi u_0(\cdot,\xi)>0$ for all $\xi$, then there exists a constant $K=K(\xi,T)$ such that $\partial_\xi u\ge \lambda e^{Kt}$ on $U_T$ for all $\xi$.
\end{theorem}

\begin{proof}
  We first consider the case $\partial_\xi u_0 \ge 0$; the case $\partial_\xi u_0\le 0$ follows by replacing $\xi$ with $-\xi$. Define $w=w(t,x;\xi):=e^{Mt}\partial_\xi u(t,x,\xi)$. Then $w$ satisfies
  \[
    Lw:= \partial_t w + g'(u)\cdot \nabla_x w + (g''(u)\cdot \nabla_x u + M) w - \epsilon \Delta_x w = 0
  \]
  with initial condition $w(0,\cdot;\xi)\ge 0$, $w(0,\cdot;\xi)\not\equiv 0$. Choose $M\ge -\inf_{\bar U_T} g''(u)\cdot\nabla_x u$, which is always possible, so that $c(t,x;\xi) = g''(u)\cdot\nabla_xu + M \ge 0$. The strong maximum principle for parabolic operators \cite{evans} then implies that $w(t,x;\xi)>0$ on $U_T$ for all $\xi$. Hence $\partial_\xi u > 0$ throughout.

  For the bound, suppose $w(0,x;\xi) \ge \lambda>0$. Let $K'=K'(\xi,T) \ge \|c(\cdot;\xi)\|_{C^0(\bar U_T)}$. Then,
  \[
    L(w-\lambda e^{-K't}) = \lambda (K'-c) e^{-K't} \ge 0.
  \]
  Since $\left.(w-\lambda e^{-K't})\right|_{t=0} \ge 0$, the weak maximum principle \cite{evans} ensures $w-\lambda e^{-K't} \ge 0$ on $U_T$. Recalling $w=e^{Mt}\partial_\xi u$, we thus obtain $\partial_\xi u(t,x,\xi) \ge \lambda e^{-(K'+M)t}$.
\end{proof}

We now revisit Example~\ref{ex:nonpos}. Consider the 1D Burgers equation with $d=1$, flux $g(u)=u^2/2$, and an initial condition of separable form (referred to as Case~\textbf{I}):
\begin{equation} \label{eq:burgers}
\left\{ \
  \begin{aligned}
    & \partial_t u + u\partial_x u = \epsilon \partial_x^2 u, \quad (t,x,\xi)\in \mathbb R_+ \times \mathbb R/[0,2\pi) \times \mathbb R, \\
    & u(0,x,\xi) = s(x) + \xi,
  \end{aligned}
\right.
\end{equation}
with $s(x)$ being $2\pi$-periodic, e.g.\ $s(x)=\sin x$. Then Theorem~\ref{thm:monotonicity} applies obviously. Without loss of generality, assume $\mathbb E[\xi]=0$. It is straightforward to verify that
\begin{equation} \label{eq:uburgers}
  u(t,x,\xi) = u(t,x-\xi t,0) + \xi,
\end{equation}
which does not hold for general fluxes $g$. Differentiating with respect to $\xi$ yields
\begin{equation} \label{eq:f_anal_bur} 
  \partial_\xi u = 1-t \left. \partial_x u(t,x',0) \right|_{x'=x-\xi t},
\end{equation}
where $\xi=\Xi_{t,x}^{-1}(v)$ is uniquely determined from $u(t,x-\xi t,0)=v-\xi$. Consequently, (\ref{eq:unc_anal}) reduces to
\begin{equation} \label{eq:unc_anal_burg}
  m(t,x,v) = \left.\partial_x^2 u(t,x,\xi)\right|_{\xi=\Xi_{t,x}^{-1}(v)} = \left.\partial_x^2 u(t, x',0) \right|_{x'=x-\Xi_{t,x}^{-1}(v) t}.
\end{equation}

\begin{proposition} \label{prop:unc_case_1}
   For the model problem (\ref{eq:burgers}), $m(\cdot,v)$ in (\ref{eq:unc_anal_burg}) is periodic in $v$ with period $2\pi/t$ and satisfies
  \[ m(t,x,v) = m(t,0,v-\frac{x}{t}), \quad t>0. \]
\end{proposition}
\begin{proof}
  Since $u$ is $2\pi$-periodic in $x$, (\ref{eq:uburgers}) implies $u(t,x,\xi+2\pi/t) = u(t,x,\xi) + 2\pi/t$, hence $\Xi_{t,x}^{-1} (v+2\pi/t) = \Xi_{t,x}^{-1}(v) + 2\pi/t$. Substituting into (\ref{eq:unc_anal_burg}) and using the $x$-periodicity of $\partial_x^2 u$, we obtain
  \[
    m(t,x,v+2\pi/t) = \left.\partial_x^2 u(t,x',0) \right|_{x'=x-\Xi_{t,x}^{-1}(v)t - 2\pi} = m(t,x,v).
  \]

  For the second equality, we show that
  \[
    U_{t,0}^{-1}\left( v-\frac{x}{t} \right) = \Xi_{t,x}^{-1}(v) - \frac{x}{t}.
  \]
  Let $\xi = U_{t,0}^{-1}\left( v-x/t \right)$. It is seen that $v-x/t = u(t,0,\xi) = u(t,-\xi t,0)+\xi$, which implies $v = u(t,x-(\xi+x/t)t,0) + \xi+x/t = u(t,x,\xi+x/t)$, and thus $\xi+x/t = \Xi_{t,x}^{-1}(v)$. Therefore, we have
  \[
  \begin{aligned}
    m(t,x,v) = \left.\partial_x^2 u(t, x',0) \right|_{x'=x-\Xi_{t,x}^{-1}(v) t}
    =\left.\partial_x^2 u(t, x',0) \right|_{x'=0-U_{t,0}^{-1}(v-\frac{x}{t}) t}
    = m(t,0,v-\frac{x}{t}).
  \end{aligned}
  \]
  This completes the proof.
\end{proof}

\subsection{Second type of random initial data} \label{subsec:case_2}

Now restrict ourselves to the 1D Burgers equation. Consider
\begin{equation} \label{eq:burgers_2}
\left\{ \
  \begin{aligned}
    & \partial_t u + u\partial_x u = \epsilon \partial_x^2 u, \quad (t,x,\xi)\in \mathbb R_+ \times \mathbb R/[0,2\pi) \times \mathbb R, \\
    & u(0,x,\xi) = \xi s(x)
  \end{aligned}
\right.
\end{equation}
with a periodic profile $s(x)=s(x+2\pi)$. We establish the following property.

\begin{proposition} \label{prop:case2}
  For system (\ref{eq:burgers_2}), if $s(x)=-s(2\pi-x)$ for all $x$ and $s(x)$ does not change sign on $(0,\pi)$, then the solution $u(t,x,\xi)$ is monotone in $\xi$ for all $t,\xi$ and $x\in(0,\pi)\cup(\pi,2\pi)$.
\end{proposition}
\begin{proof}
  First note that $u(t,0,\xi)=u(t,\pi,\xi)=0$ for all $(t,\xi)$, since both the initial data and the governing Burgers equation preserve the symmetry $u(x,\cdot)=-u(-x,\cdot)=-u(2\pi-x,\cdot)$. Assume without loss of generality that $s>0$ on $(0,\pi)$. Then, by arguments analogous to those in the proof of Theorem~\ref{thm:monotonicity}, we consider $w=w(t,x;\xi):=e^{Mt}\partial_\xi u$, which satisfies a parabolic equation on $(t,x)\in[0,T]\times[0,\pi]$ with boundary conditions $w\equiv 0$ at $x=0,\pi$ (since $u\equiv 0$ at these points). By the strong maximum principle, we immediately deduce $w>0$ on $(t,x)\in(0,T]\times(0,\pi)$, and hence $w<0$ on $(t,x)\in(0,T]\times(\pi,2\pi)$.
\end{proof}

\begin{remark}
  In analogy with Theorem~\ref{thm:monotonicity}, one can further derive an exponentially decaying lower bound for $\partial_\xi u$ on $(0,\pi)$. Let $\lambda_1\in\mathbb R$ denote the principal eigenvalue of the operator $A:=- \epsilon \partial^2_x + u\partial_x + (\partial_x u + M) $ and let $\varphi$ be the corresponding positive eigenfunction on $(0,\pi)$. Recall that $w=e^{Mt}\partial_\xi u$ and $L=\partial_t+ A$. Then we have $L(w-\epsilon e^{-Kt}\varphi) = \epsilon (K-\lambda_1)\varphi e^{-Kt} \ge 0$ by choosing $K\ge \lambda_1$. Moreover, since $w(0,\cdot)-\epsilon \varphi \ge 0$ for sufficiently small $\epsilon$, due to $\partial_\nu w(0,\cdot)<0$ (negative outward normal derivative on the boundary), the weak maximum principle yields $\partial_\xi u\ge \epsilon e^{-(K+M)t}\varphi$.
\end{remark}

Therefore, under the assumptions of Proposition~\ref{prop:case2}, the mapping $\Xi_{t,x}:\xi\mapsto v=u(t,x,\cdot)$ is invertible for all $x\in(0,2\pi)\setminus\{\pi\}$, with the CDF and PDF given by (\ref{eq:F_anal}) and (\ref{eq:f_anal}), respectively. By Oleinik's inequality $\partial_x u<1/t$, we obtain $u(t,x,\xi)<x/t$, which implies $f(t,x,v)=0$ for all $v\ge x/t$, $t>0$ and $x\in (0,\pi)$. Furthermore, it holds that $f(t,0,v)=f(t,\pi,v)=\delta(v)$ for all $(t,v)$. For the `unclosed' term in (\ref{eq:unc_anal_burg}), it is straightforward to verify that $m(0,x,v) = -v$ for $x\ne k\pi$.

\begin{remark} \label{rem:cole_hopf}
The 1D Burgers equation in (\ref{eq:burgers}) and (\ref{eq:burgers_2}) admits a closed-form solution via the Cole-Hopf transformation:
\begin{align} 
  u(t,x,\xi) &= -2\epsilon \frac{\partial_x \phi}{\phi}, \label{eq:u_cole_hopf} \\
  \phi = \phi(t,x;\xi) &= \frac{1}{2\pi} A_0^{[\epsilon,\xi]} + \sum_{k=1}^\infty \frac{e^{-\epsilon k^2 t}}{\pi} \left( A_k^{[\epsilon,\xi]}\cos kx + B_k^{[\epsilon,\xi]} \sin kx \right), \label{eq:phi_cole_hopf}
\end{align}
with
\[
  \begin{bmatrix} A_k^{[\epsilon,\xi]} \\  B_k^{[\epsilon,\xi]} \end{bmatrix} = \int_0^{2\pi} \phi_0 (x;\epsilon,\xi) \begin{bmatrix} \cos kx \\ \sin kx \end{bmatrix} dx, 
  \quad
  \phi_0(x;\epsilon,\xi) = \exp \left[ -\frac{1}{2\epsilon} \int_0^x u(0,x',\xi)\,dx' \right].
\]
Differentiation of (\ref{eq:u_cole_hopf}) gives $\partial_x^2 u = -2\epsilon \frac{\partial_x^3\phi}{\phi} + \frac{3}{2\epsilon} u \partial_x u - \frac{u^3}{4\epsilon^2}$, where $\partial_x u$ and higher derivatives of $\phi$ can be obtained analogously. These formulas allow semi-analytical computation of the statistical quantities $f(t,x,v)$, $F(t,x,v)$, and $m(t,x,v)$ in (\ref{eq:burgers}). For given $(t,x,v)$, the corresponding $\xi=\Xi_{t,x}^{-1}(v)$ is located by solving $u(t,x,\xi)=v$ (e.g.\ via the bisection method), after which (\ref{eq:F_anal}), (\ref{eq:f_anal}), and (\ref{eq:unc_anal_burg}) apply.
\end{remark}

\section{Numerical results} \label{sec:numerical}

In this section, we present numerical results for the two 1D cases with exact solutions constructed in the previous section. Our goal is to approximate the statistical conservation law for $\hat F(t,x,v)$ by implementing Algorithm~\ref{alg:scle}. In particular, for the Burgers equation, the samples $u_i$ and $(\partial_x^2u)_i$ used in (\ref{eq:estimator}) are taken from the analytical solutions (\ref{eq:u_cole_hopf}) and (\ref{eq:phi_cole_hopf}) with $\xi=\xi_i$ for all $(t,x)$. Moreover, the (semi-)analytical solutions for PDFs/CDFs in both cases (with the 1D Burgers equation) are computed based on the procedures outlined in Remark~\ref{rem:cole_hopf}.

\subsection{Case I in (\ref{eq:burgers})} \label{subsec:test_1}

We consider the setting $\epsilon = 0.1$, $s(x)=\sin x$, and $p(\xi) = \mathcal N(0,\sigma^2)$ with $\sigma=0.5$ (a Gaussian distribution with mean $0$ and standard deviation $0.5$). The same configuration was studied in \cite{cho2014}. Figure~\ref{fig:f_anal} displays the exact solutions for the PDF $f(t,\pi,v)$. In addition, we compute the expectation and standard deviation of $u$ over the full domain as
\begin{align}
  \mathbb E_\xi [u](t,x)&=\int u(t,x,\xi)p(\xi)\,d\xi = \int u(t,x-\xi t,0)p(\xi) \, d\xi, \label{eq:Eu_anal} \\
  SD_\xi[u](t,x) &= \sqrt{\mathbb E_\xi[u^2](t,x)-\left(\mathbb E_\xi[u](t,x)\right)^2}, \label{eq:SD_anal}
\end{align}
with results shown in Figure~\ref{fig:Eu}. These plots illustrate the dissipative nature of the system: $\mathbb E_\xi [u]\to 0$ for all $x$ as $t\to\infty$, while the standard deviation exhibits more complex dynamics. In particular, Figure~\ref{fig:f_anal} shows that the profile $f(t,x=\pi,v)$ develops a bimodal structure for $t\le 2$, which then sharpens into a multi-spike pattern ($t\le 20$). In the long-time limit $t\to\infty$, we have $u\to \xi$ for all $x,\xi$, and thus $f(v)\to p(v)$ for all $x$ (as confirmed by the $t=100$ result in Figure~\ref{fig:f_anal}).

\begin{figure}[htbp]
  \centering
  \includegraphics[height=8cm]{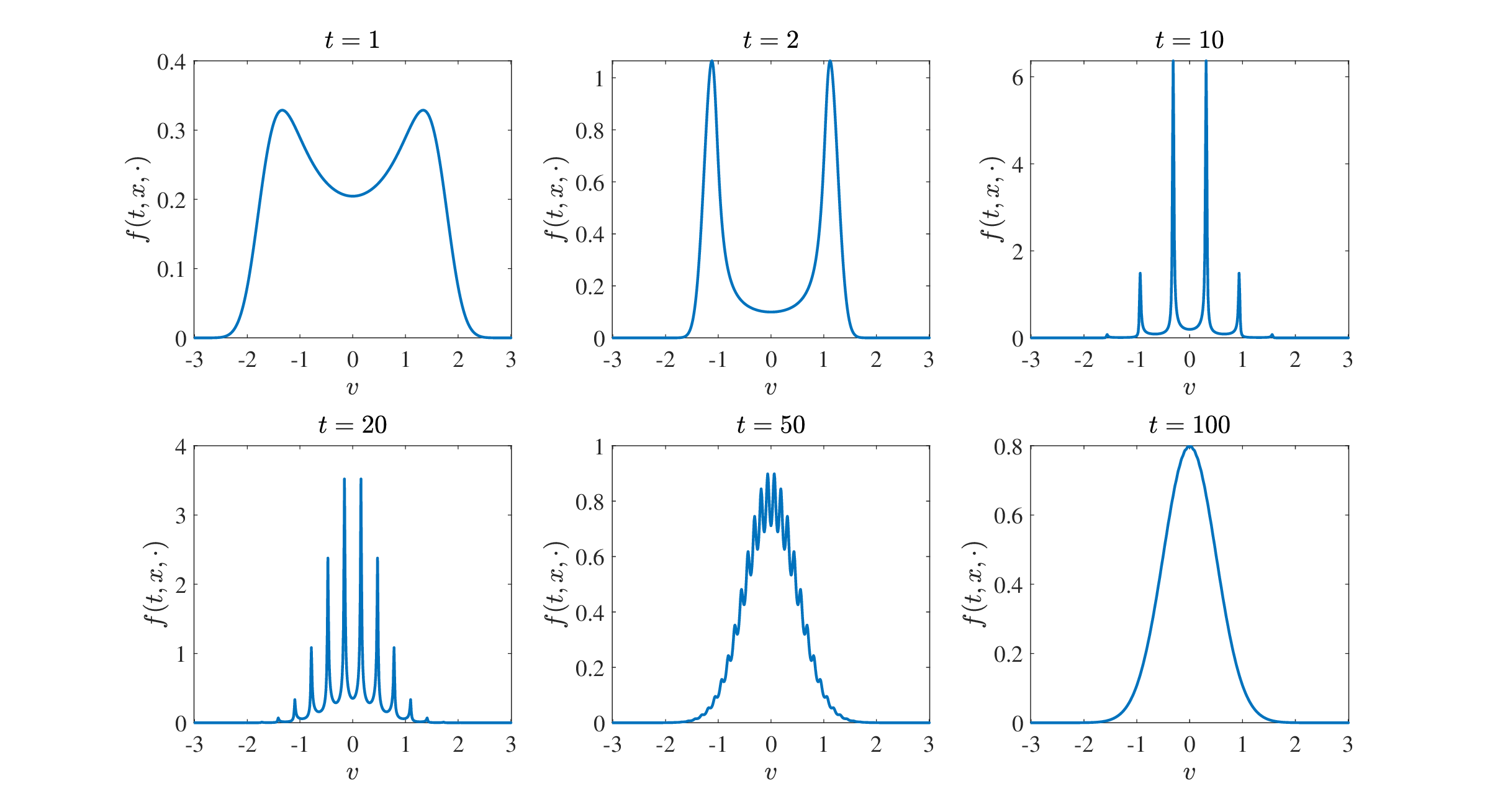}
  \caption{Analytical solutions for $f(t,x=\pi, v)$ at different values of $t$ for Case \textbf{I} in (\ref{eq:burgers}). The profiles are obtained based on the procedure outlined in Remark~\ref{rem:cole_hopf}.}
  \label{fig:f_anal}
\end{figure}

\begin{figure}[htbp]
  \centering  
  \includegraphics[height=4.5cm]{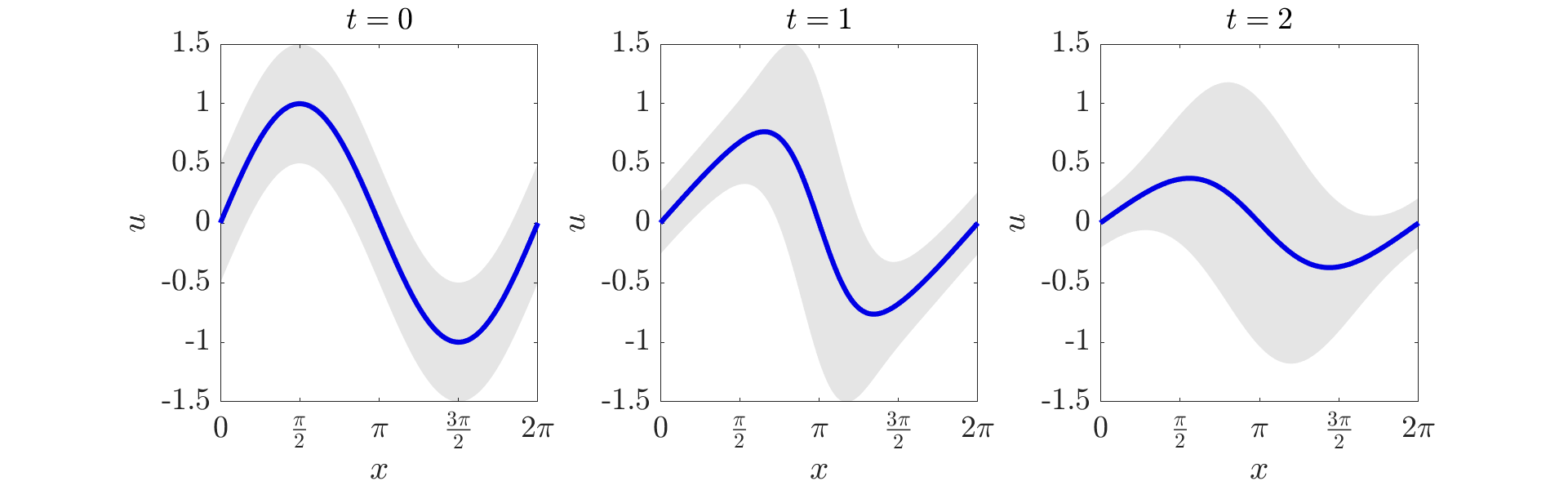}
  \caption{Analytical solutions of $\mathbb E_\xi[u]$ (solid lines) and $SD_\xi[u]$ (standard deviation, shadowed regions) at different values of $t$ for Case \textbf{I} in (\ref{eq:burgers}). The results are computed using (\ref{eq:Eu_anal}) and (\ref{eq:SD_anal}).}
  \label{fig:Eu}
\end{figure}

With the analytical solutions in hand, we now approximate the statistical conservation law with (\ref{eq:Fappr}). First, the estimator $\hat m_N(t,x,v)$ is constructed using (\ref{eq:estimator}) and (\ref{eq:kerconst}) with $h_v = 0.01$ and varying numbers $N$ of samples. Then, (\ref{eq:Fappr}) is solved through (\ref{eq:charc_de_p}) and (\ref{eq:Fhat_char_sol}). The resulting PDFs and CDFs at $x=\pi$ for various $t$ and $N$ are shown in Figure~\ref{fig:fF}. The case $N=0$ corresponds to solutions of (\ref{eq:F_eq_noeps}) given by (\ref{eq:FS0_case1}). As revealed in Example~\ref{ex:nonpos}, $f_{S0}$ becomes negative for $t>1$, although it captures the bimodal structure for $t\le 1$ and even reproduces the multiple `spikes' of $f(v)$ for $t>1$. This suggests that the spikes originate from the convective term in (\ref{eq:F_eq_noeps}), while the drift term involving $m(t,x,v)$ corrects the negativity. Increasing $N$ substantially improves the accuracy of $F(t,x,v)$ (see the right column of Figure~\ref{fig:fF}): for $N=100$, the PDFs (as $v$-derivatives of $F$) remain scattered and can even become negative; for $N=10^3$ and $5\times 10^4$, the computed CDFs and PDFs closely match the exact solutions up to $t=20$.

Figure~\ref{fig:dF} further quantifies the $L^2_v$ errors between the computed $F(t,x,v)$ and the exact solution (\ref{eq:F_anal}). Several observations can be made. First, the errors decrease significantly as $N$ increases. Since all results use a fixed $h_v=0.01$, the errors are \emph{not} expected to vanish as $N\to\infty$, because $\hat m_N(t,x,v)$ does not converge to the exact $m(t,x,v)$ in that limit (see Section~\ref{sec:approx}), which is clearly evident for all $t$. Second, the errors increase with $t$: for $N=10^4$, the error grows from $\sim 2\times 10^{-4}$ at $t=1$ to $O(10^{-3})$ at $t=20$. This behavior can be attributed both to the bounds in Theorem~\ref{thm:error}, which grow substantially with $t$, and to larger relative errors in estimating $m(t,x,v)$ (see Figure~\ref{fig:uxxcondu} later).

It is instructive to compare these errors with those from the Monte Carlo approach, i.e., direct reconstruction of CDFs from the samples $\{u_i\}$. These results, also plotted in Figure~\ref{fig:dF}, always exhibit a convergence rate of $O(N^{-1/2})$. Asymptotically (as $N\to\infty$), the statistical conservation law (\ref{eq:Fappr}) yields errors no better than $O(N^{-2/5})$ (see (\ref{eq:mise}) and Theorem~\ref{thm:error}), which is indeed larger than the Monte Carlo rate. Nevertheless, for many finite $N$, the approximated CDFs from (\ref{eq:Fappr}) perform better, demonstrating the practical value of statistical conservation laws with finitely many samples. The advantage of the PDE structure is particularly evident in the extreme case $N=0$: while it is impossible to reconstruct $F$ directly from samples, the statistical law (\ref{eq:F_eq_noeps}) still provides meaningful, albeit crude, information through $F_{S0}(t,x,v)$.

\begin{figure}[htbp]
  \centering
  \subfloat
  {\includegraphics[height=4.8cm]{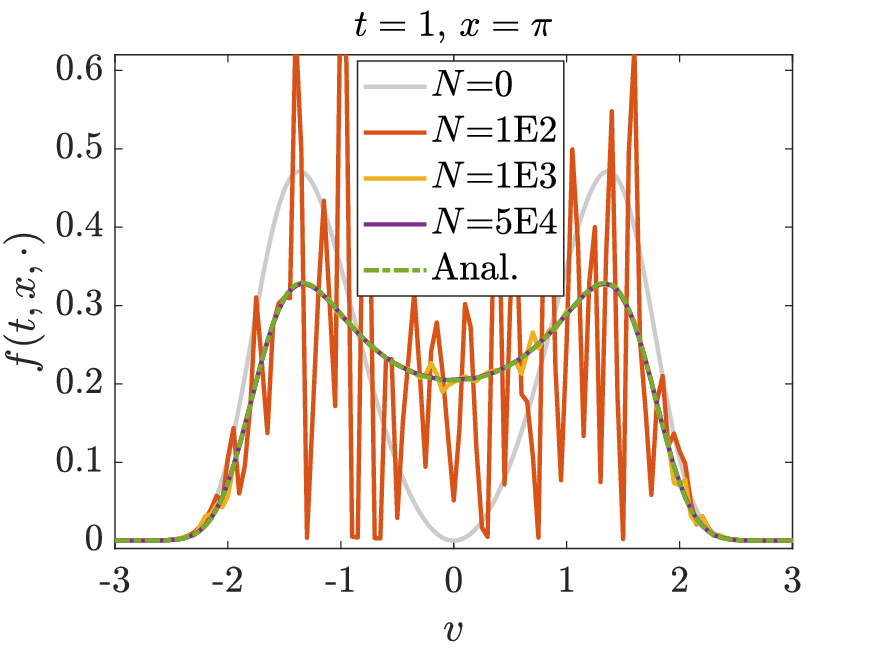}}
  \subfloat
  {\includegraphics[height=4.8cm]{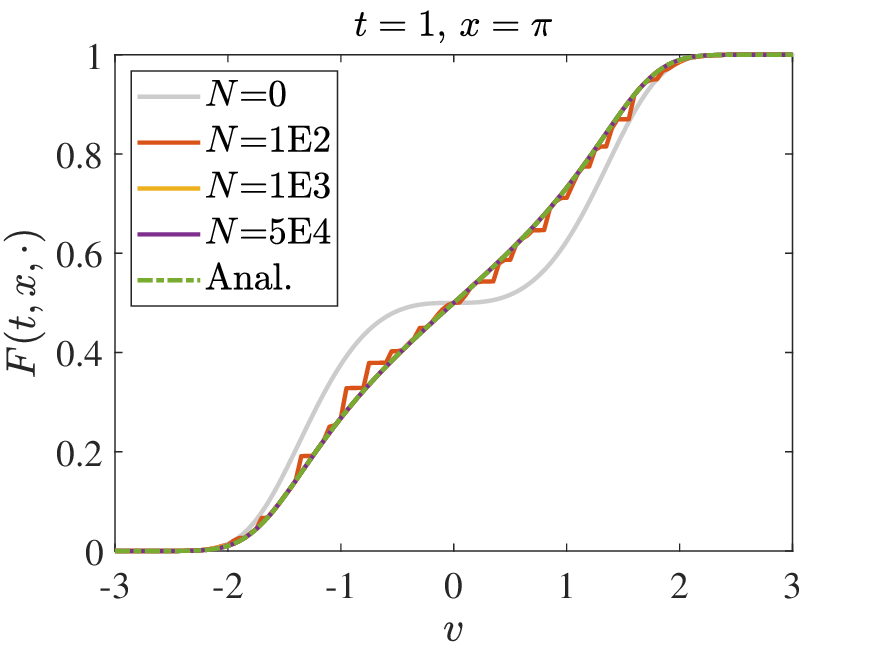}}
  \\
  \subfloat
  {\includegraphics[height=4.8cm]{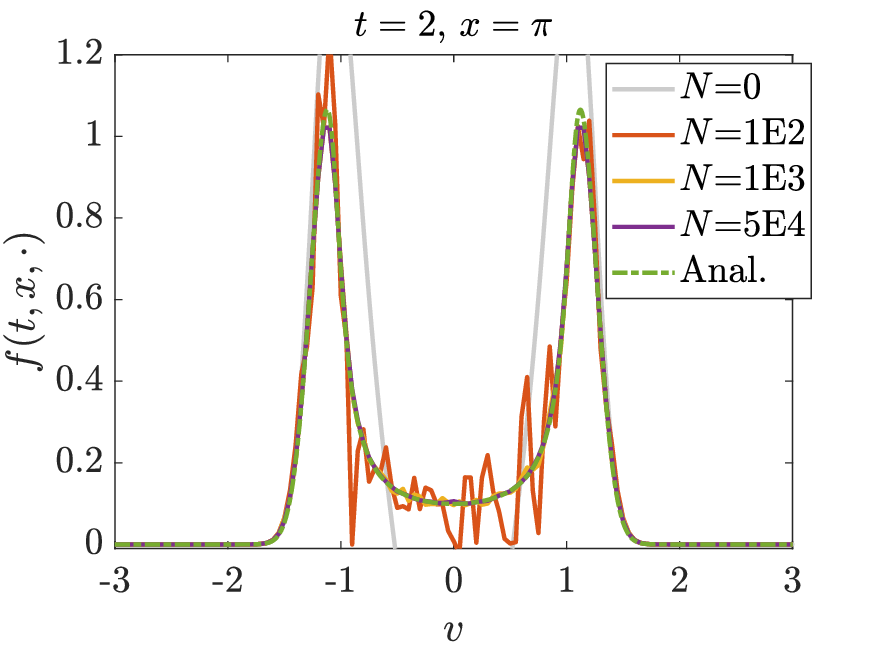}}
  \subfloat
  {\includegraphics[height=4.8cm]{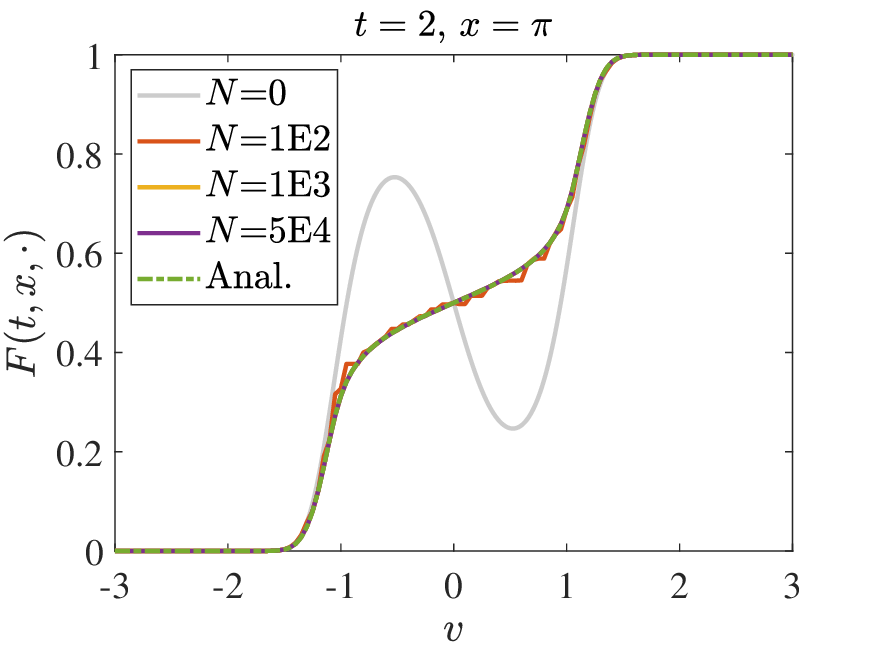}}
  \\
  \subfloat
  {\includegraphics[height=4.8cm]{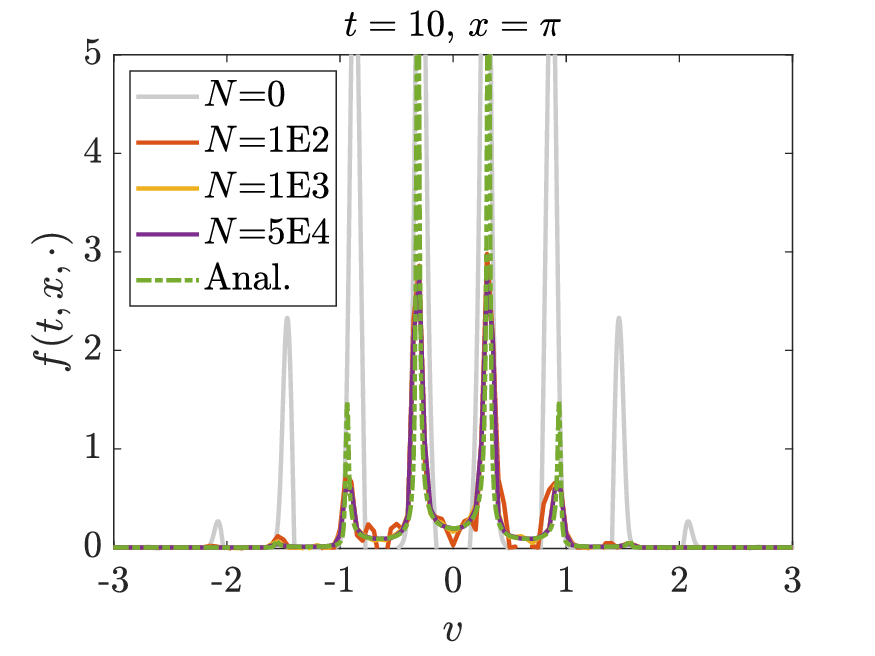}}
  \subfloat
  {\includegraphics[height=4.8cm]{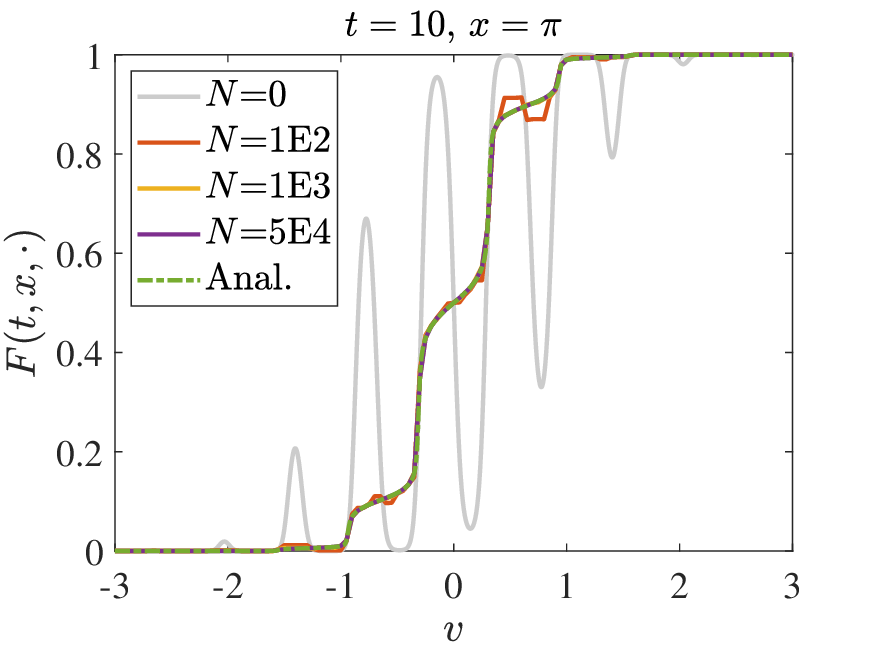}}
  \\
  \subfloat
  {\includegraphics[height=4.8cm]{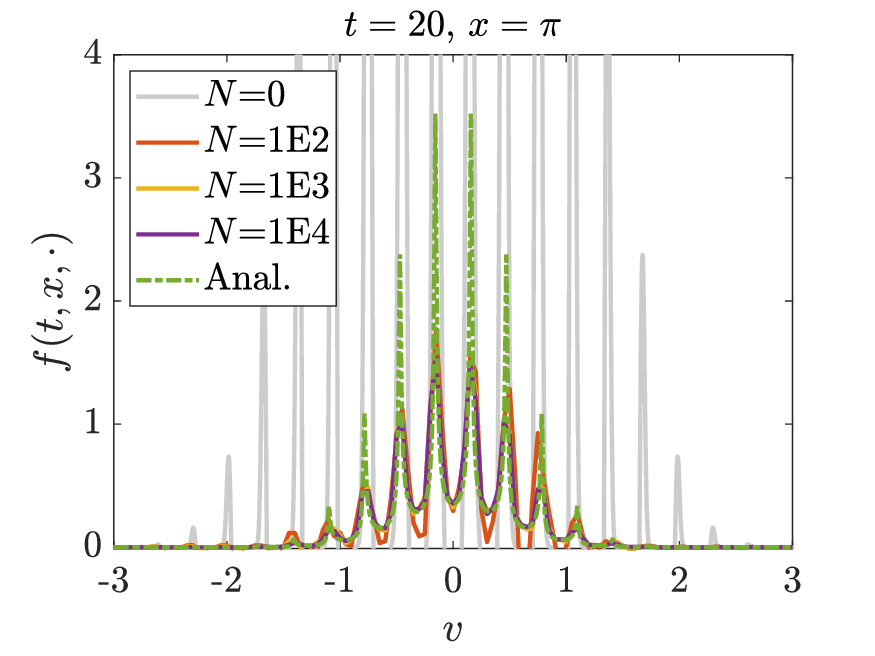}}
  \subfloat
  {\includegraphics[height=4.8cm]{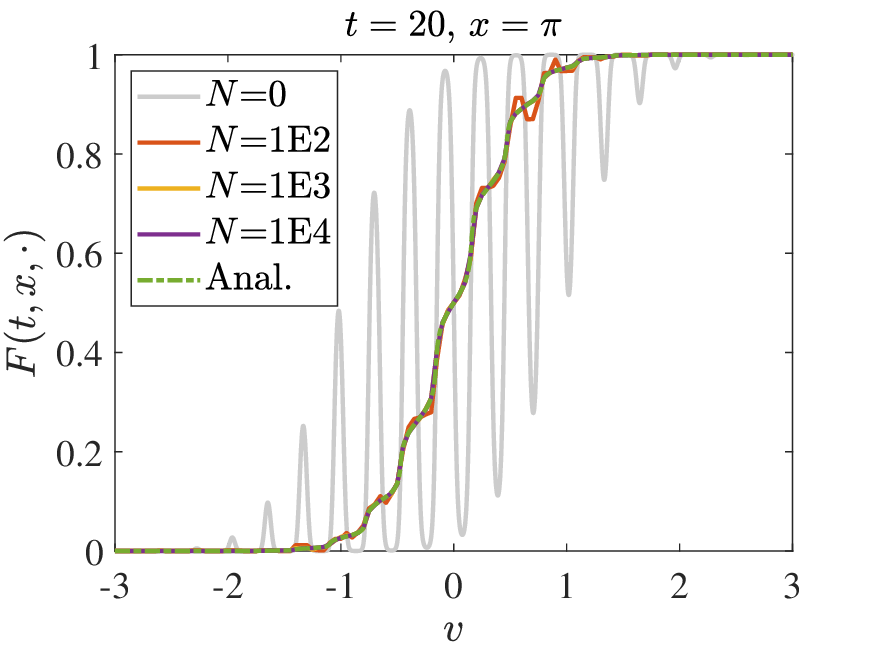}}
  \caption{Approximated PDFs (left column) and CDFs (right column) at $x=\pi$ for various times $t$ in Case \textbf{I} of (\ref{eq:burgers}). Results are computed with a fixed $h_v=0.01$ for all values of $N$. The case $N=0$ corresponds to the analytical solution of (\ref{eq:F_eq_noeps}) given in (\ref{eq:FS0_case1}). Analytical reference solutions are also included for comparison.}
  \label{fig:fF}
\end{figure}

\begin{figure}[htbp]
  \centering
  \subfloat
  {\includegraphics[height=5cm]{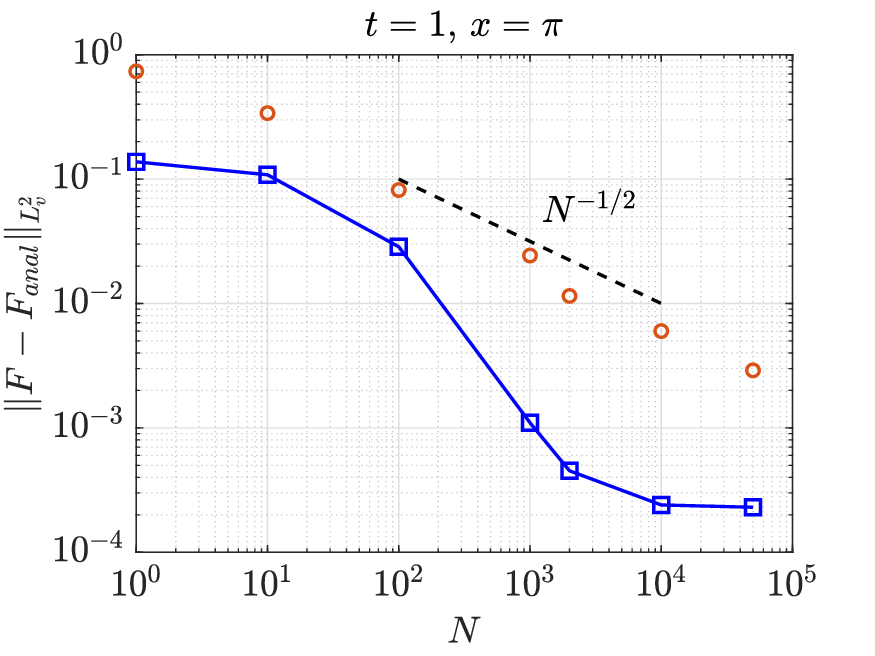}}
  \subfloat
  {\includegraphics[height=5cm]{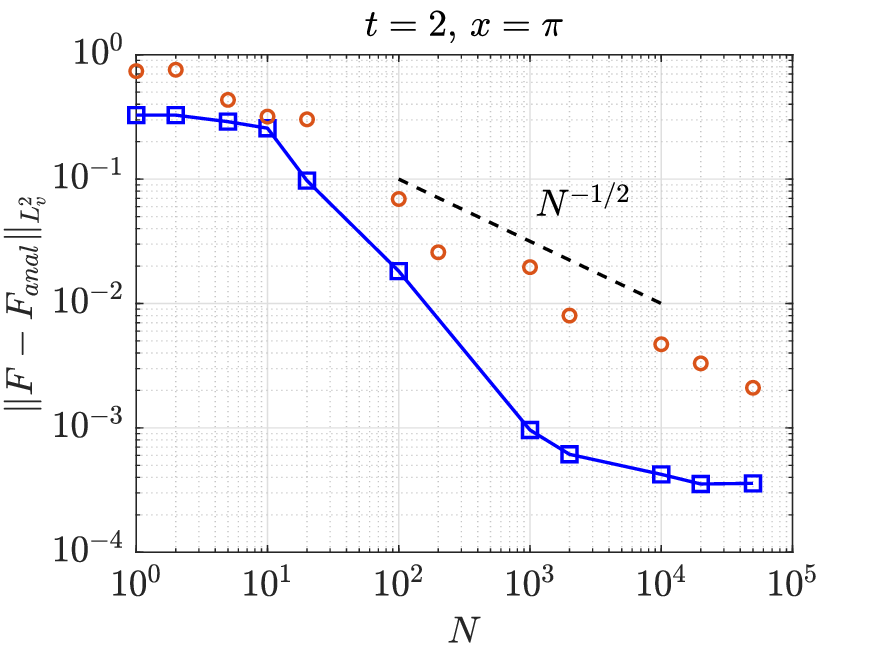}}
  \\
  \subfloat
  {\includegraphics[height=5cm]{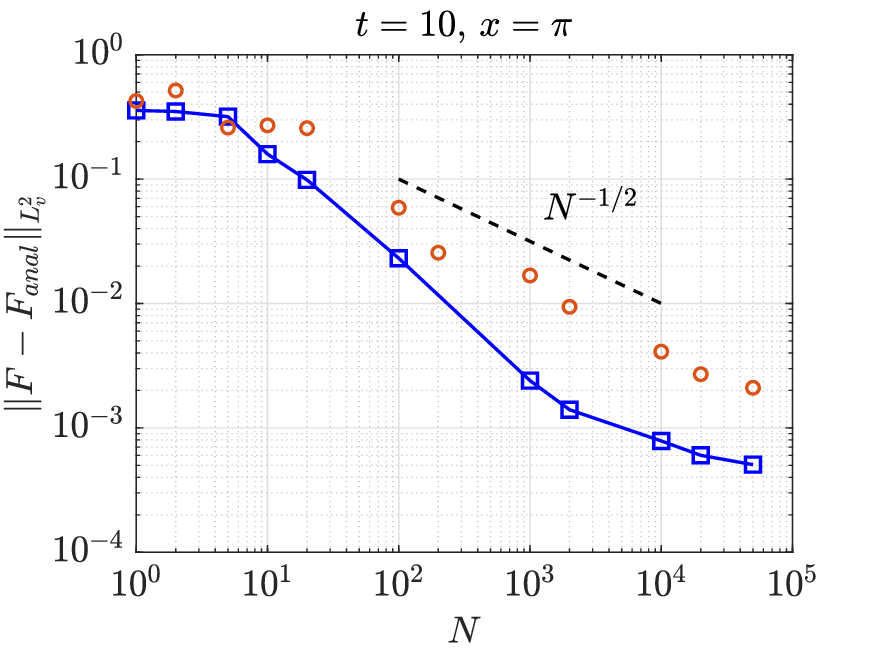}}
  \subfloat
  {\includegraphics[height=5cm]{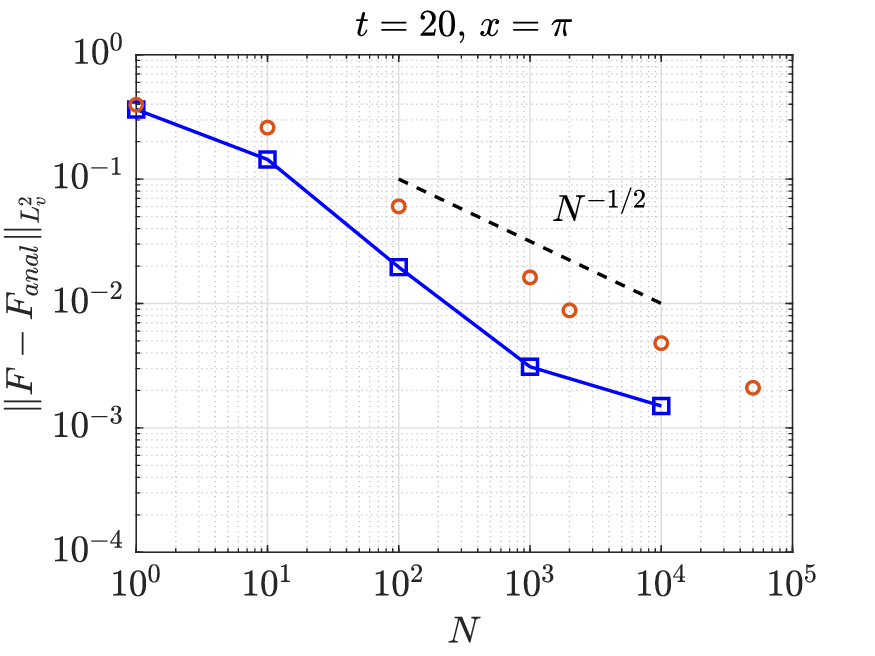}}
  \caption{$L_v^2$-errors between the approximated CDFs and the analytical solution (\ref{eq:F_anal}) at $x=\pi$ for different times $t$ in Case \textbf{I} of (\ref{eq:burgers}). Results are obtained with a fixed $h_v=0.01$ for all $N$. Red dots indicate the $L^2_v$-errors of the CDFs directly reconstructed from the samples $\{u_i\}_{i=1}^N$.}
  \label{fig:dF}
\end{figure}

Fig.~\ref{fig:uxxcondu} illustrates the accuracy of the estimator $\hat m_N(t,x,v)$ for approximating $m(t,x,v)$ with increasing $N$. The exact solutions (top row, $x=\pi$, $t=2,\,10$) confirm the properties stated in Proposition \ref{prop:unc_case_1}: periodicity in $v$, which diminishes over time. For small $N$ ($10$-$100$), $\hat m_N$ vanishes over most of the $v$-domain. Increasing $N$ to $10^4$ substantially improves the estimate near $v=0$, while large $|v|$ regions remain unresolved due to the scarcity of samples (since $\xi$ is Gaussian with zero mean). As these regions correspond to nearly zero values of $f$ (see Fig.~\ref{fig:fF}), the induced error has negligible impact. The relative errors $|\hat m_N-m|/|m|$ (bottom row, at given values of $(t,x,v)$) decrease only when $N$ is sufficiently large, showing an asymptotic trend close to $O(N^{-2/5})$ at $t=2$. At $t=10$, the errors are larger and converge more slowly, reflecting both the reduced magnitude of $m$ and the bias introduced by the fixed bandwidth $h_v$.

A complementary perspective on (\ref{eq:Fappr}) is obtained by tracking the initial points $\{(X,V)|_{\tau=0}\}$ of characteristics (\ref{eq:charc_de_p}) corresponding to the lines $\{X|_{\tau=t}=\pi,\, V|_{\tau=t}\in[-3,3]\}$. As shown in Fig.~\ref{fig:F_preimage}, these curves, exact solutions marked by yellow dots, undergo stretching, compression, and rotation during evolution. By construction, exact $m(t,x,v)$ ensures that $\{(X,V)|_{\tau=0}\}$ forms a non-decreasing path in the $G(x,v)$-landscape. Figures~\ref{fig:F_preimage}a\&b also display the approximated paths $\{(\hat X,\hat V)|_{\tau=0}\}$ for $t=1,\,2$: discrepancies are more pronounced for $N=10$ and at longer times, while $N=1000$ yields closer agreement.

\begin{figure}[htbp]
  \centering
  \subfloat
  {\includegraphics[height=5cm]{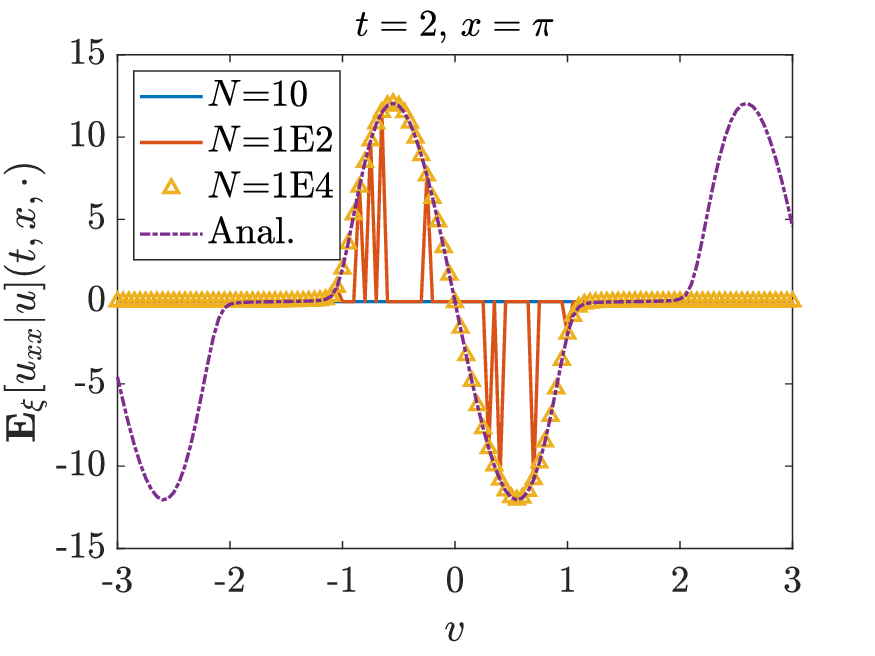}}
  \subfloat
  {\includegraphics[height=5cm]{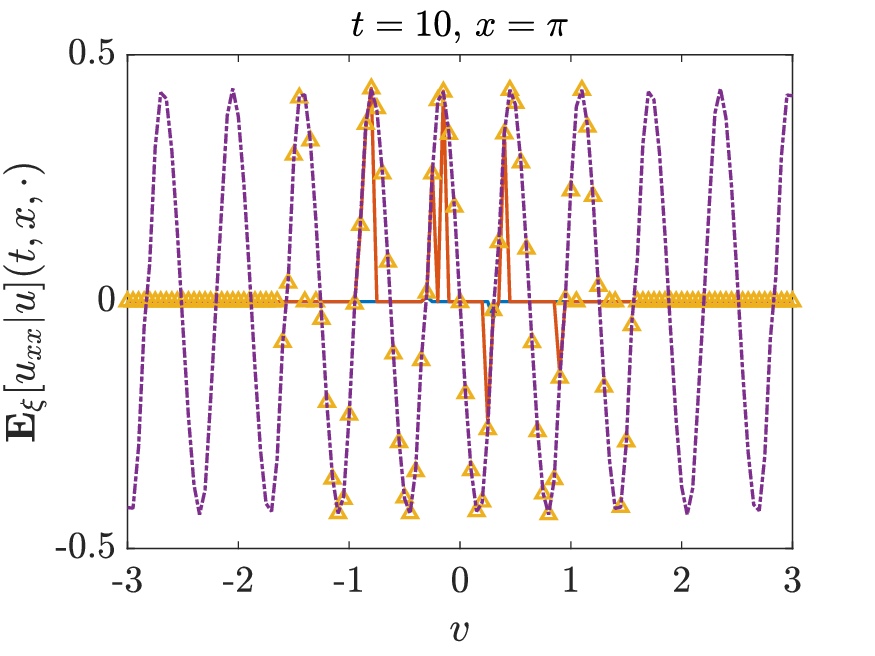}}
  \\
  \subfloat
  {\includegraphics[height=5cm]{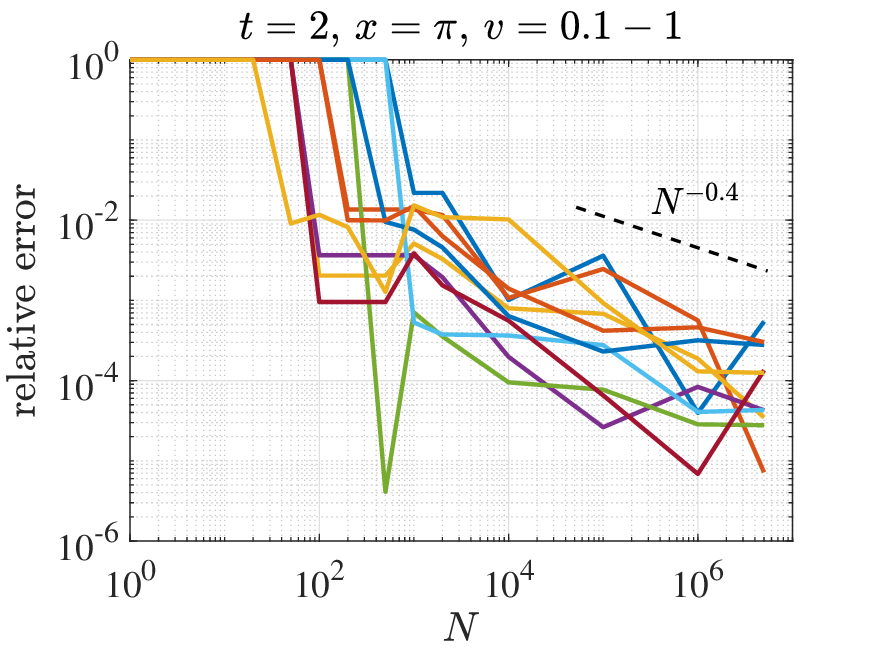}}
  \subfloat
  {\includegraphics[height=5cm]{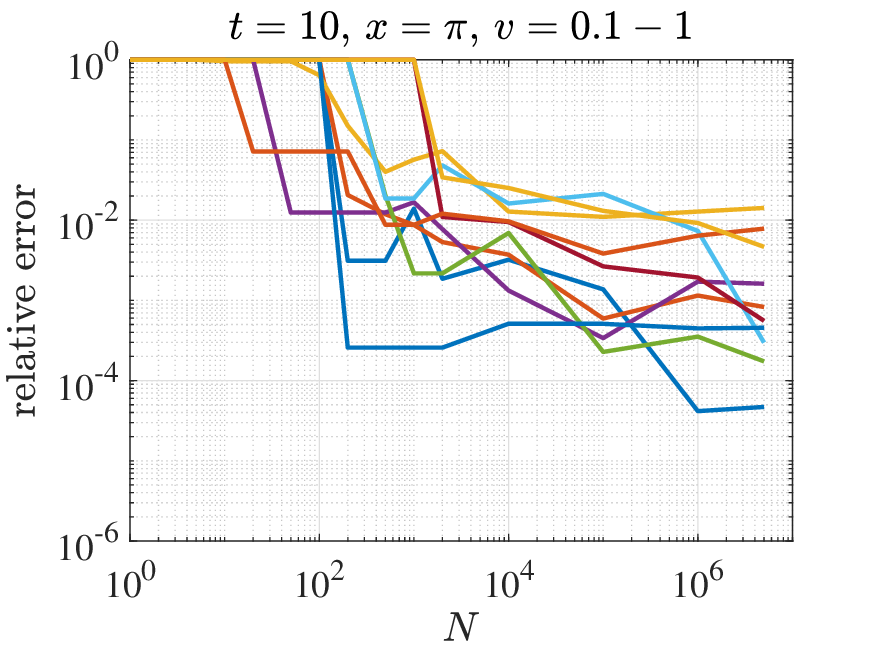}}
  \caption{Top row: numerical approximations of $m(t,x,v)=\mathbb E_\xi[\partial_x^2u|u](t,x,v)$ with different sample sizes $N$ for Case \textbf{I} in (\ref{eq:burgers}). Bottom row: relative errors $|\hat m_N-m|/|m|$ at selected $(t,x,v)$ points for varying $N$. All results are computed with a fixed $h_v=0.01$.}
  \label{fig:uxxcondu}
\end{figure}

\begin{figure}[htbp]
  \centering
  \subfloat
  {\includegraphics[height=5.6cm]{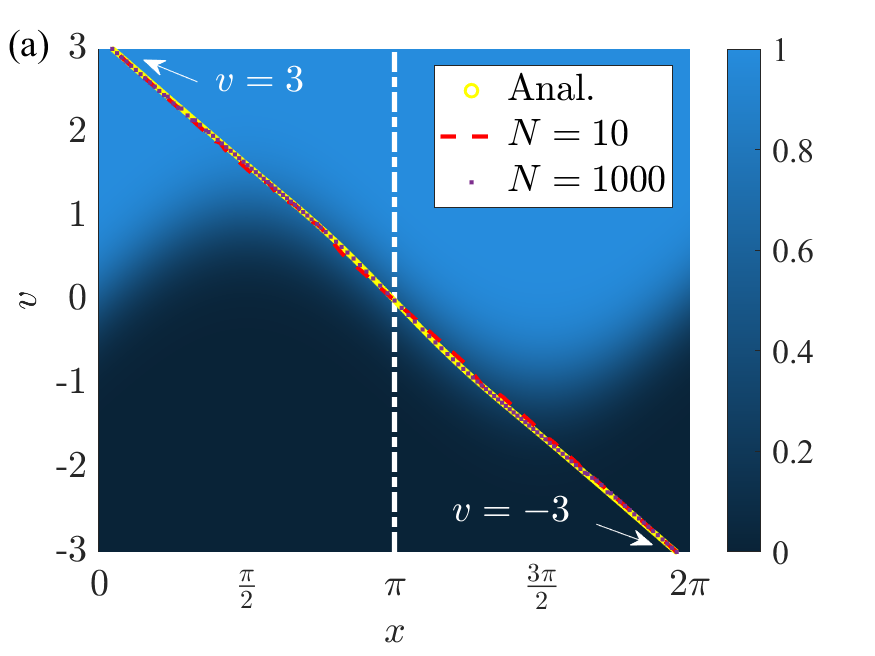}}
  \\
  \subfloat
  {\includegraphics[height=5.6cm]{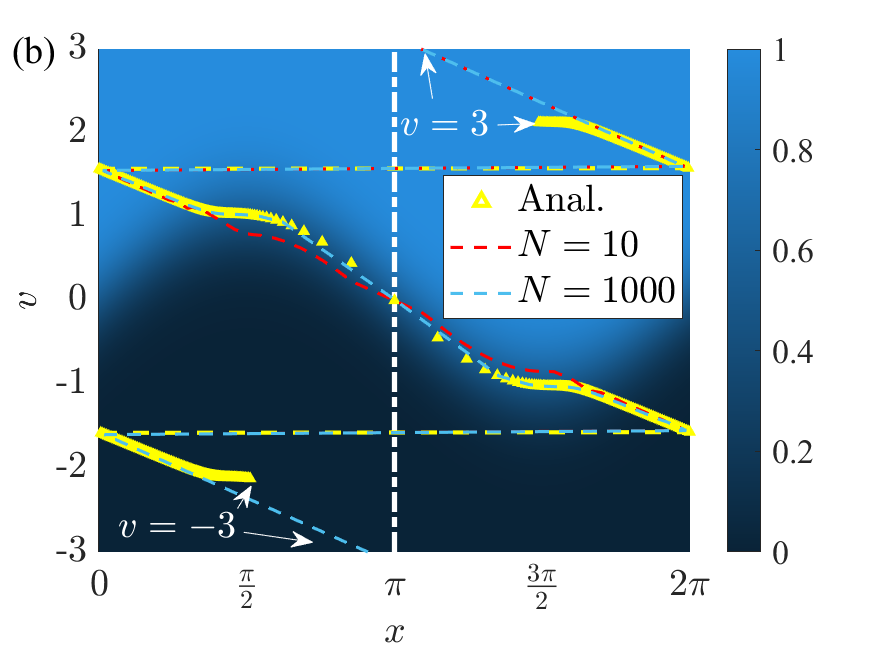}}
  \\
  \subfloat
  {\includegraphics[height=5.6cm]{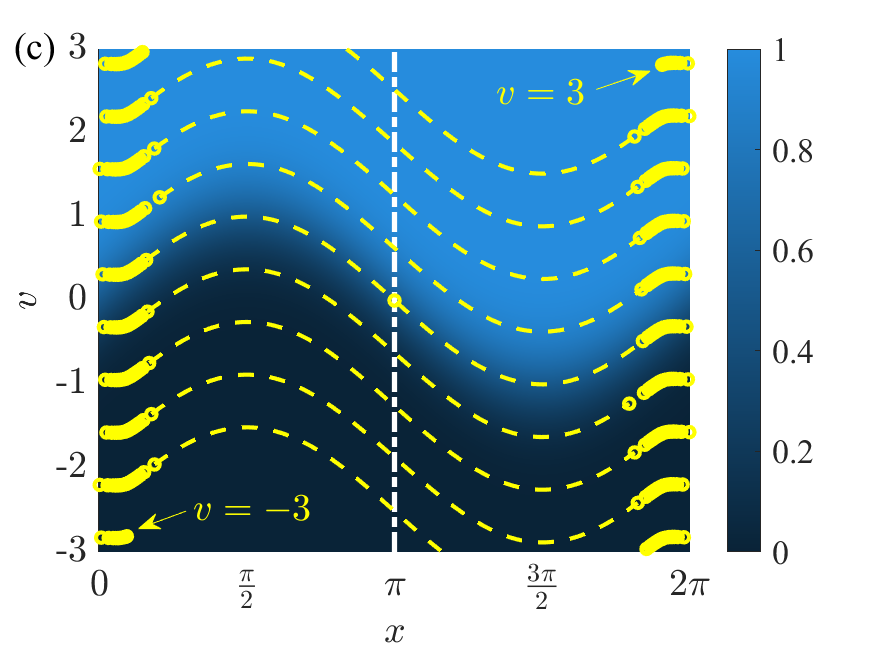}}
  \caption{Backward characteristic tracing for $\hat \psi(\tau;t,x,v)=(\hat X(\tau),\hat V(\tau))$ in (\ref{eq:charc_de_p}). The white dashed line denotes $\hat X|_{\tau=t} = \pi$ and $\hat V|_{\tau=t} \in [-3, 3]$ at times (a) $t=1$, (b) $t=2$ and (c) $t=10$. Yellow dots indicate exact initial points $\psi(0)=(X(0),V(0))$. Piecewise structures in (b) and (c) arise from the periodic boundary in $x$. Also shown in (a) and (b) are approximations $(\hat X(0),\hat V(0))$ with $N=10, \, 1000$ (using a fixed $h_v=0.01$). The background shading represents the initial data $G(x, v)$ defined in (\ref{eq:G1}), with the color scale displayed on the right.}
  \label{fig:F_preimage}
\end{figure}

Finally, Fig.~\ref{fig:hveffect} examines the effect of $h_v$ on the CDF $F(t=2,x=\pi,v)$. For each $N$, an optimal $h_v$ exists and decreases with $N$, qualitatively consistent with the MISE scaling in (\ref{eq:mise}). However, for sufficiently large $N$ (e.g., $10^4$), the influence of $h_v$ becomes negligible: the $L_v^2$-error is then dominated by discretization in solving (\ref{eq:Fappr}) with the Euler scheme tracing back-in-time the characteristics (\ref{eq:charc_de_p}). This is confirmed by the dashed line, which manifests the discretization error of the scheme solving (\ref{eq:charc_de_p}) with the exact $m(t,x,v)$ from (\ref{eq:unc_anal_burg}). The close match with the total error in $F$ under $N=10^4$ indicates that at large $N$, numerical integration error overtakes sampling error. We defer the systematic a posteriori error analysis to our future work.

\begin{figure}[htbp]
  \centering
  \includegraphics[height=5cm]{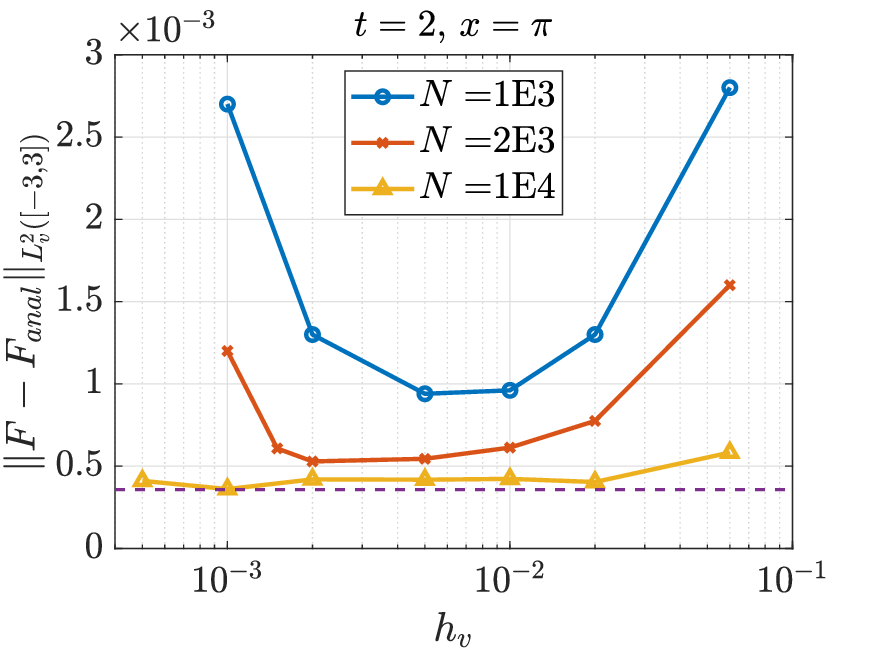}
  \caption{Influence of $h_v$ and $N$ on the $L_v^2$-errors between the approximated CDF $\hat F(t=2,x=\pi,v)$ and the analytical solution for Case \textbf{I} in (\ref{eq:burgers}). The dashed line shows the error of $\hat F$ computed by solving (\ref{eq:charc_de_p}) with the exact $m(t,x,v)$ from (\ref{eq:unc_anal_burg}), thus representing the discretization error of the scheme.}
  \label{fig:hveffect}
\end{figure}

\subsection{Case II in (\ref{eq:burgers_2})}

Here we specify $\epsilon = 0.1$, $s(x)=\sin x$, and $p(\xi)$ as the Gamma distribution $p(\xi) = \xi^{k-1}e^{-\xi/\theta}\mathbf{1}_{\xi>0}/(\Gamma(k)\theta^k)$ with $k=2.8$ and $\theta=0.18$. Denote $\Phi(x)=\int_{-\infty}^x p(x')\,dx'$. The initial condition corresponding to (\ref{eq:Fappr}) for $x\in (0,\pi)$ then reads $G(x,v)=F(0,x,v)= \Phi\left( \frac{v}{\sin x} \right)$.

Figure~\ref{fig:ctime_f_F_anal} shows the analytical solutions of $f$ (via (\ref{eq:f_anal})) and $F$ (via (\ref{eq:F_anal})) at $t=1,2$ for different values of $x\in(0,\pi)$. Starting from Gamma-distributed inputs, most $f(v)$ profiles are characterized by a single peak. For comparison, the figure also includes the solutions of (\ref{eq:F_eq_noeps}), i.e., the `zero-sample' approximations, which read for $v<x/t$:
\begin{subequations} \label{eq:FS0_case2}
\begin{align}
  F_{S0}(t,x,v) &= G(x-vt,v) = \Phi\left( \frac{v}{\sin(x-vt)} \right), \\
  f_{S0}(t,x,v) &= p\left( \frac{v}{\sin(x-vt)} \right) \cdot \frac{\partial}{\partial v}\left( \frac{v}{\sin(x-vt)} \right).
\end{align}
\end{subequations}
In this case, $f_{S0}$ remains positive. The insets report the $L_v^2$-errors of $f$ and/or $F$ at $t=1,2$ for different $x$. The `zero-sample' results remain close to the analytical solutions for small $x$, but errors grow rapidly as $x$ approaches $\pi$. This is further illustrated in Fig.~\ref{fig:ctime_Eu}, where the expectations $\mathbb E_\xi[u](t,x)$ at $t=1,2$ are plotted. The largest deviation occurs near $x=\pi$, where the `zero-sample' approximation significantly overestimates the expectation. Here the analytical results (solid lines) are obtained from (\ref{eq:Eu_anal}), while the `zero-sample' results (dashed lines) are computed as $\int v f_{S0}(\cdot,v)\,dv$.

\begin{figure}[htbp]
  \centering
  \subfloat
  {\includegraphics[height=5cm]{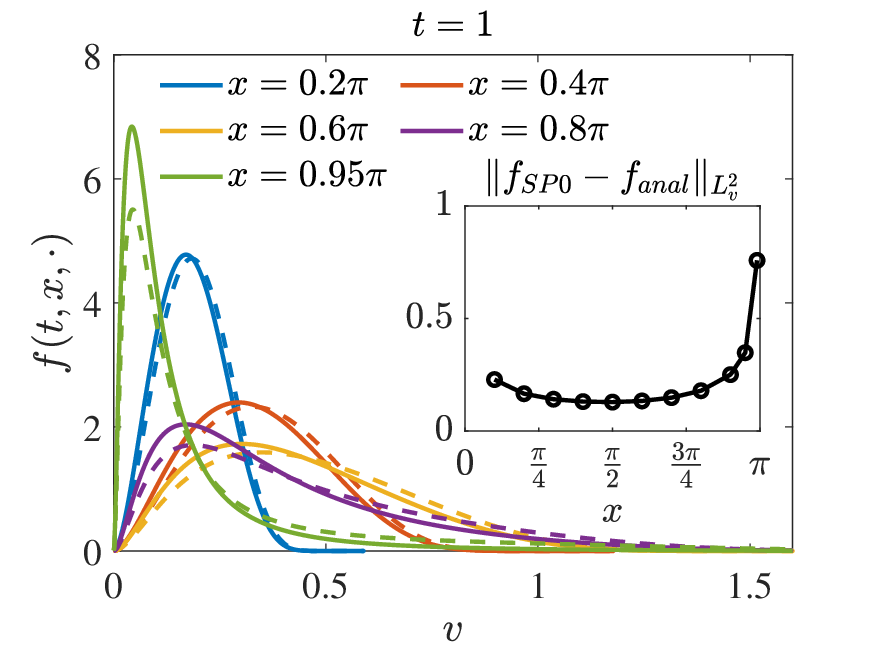}}
  \subfloat
  {\includegraphics[height=5cm]{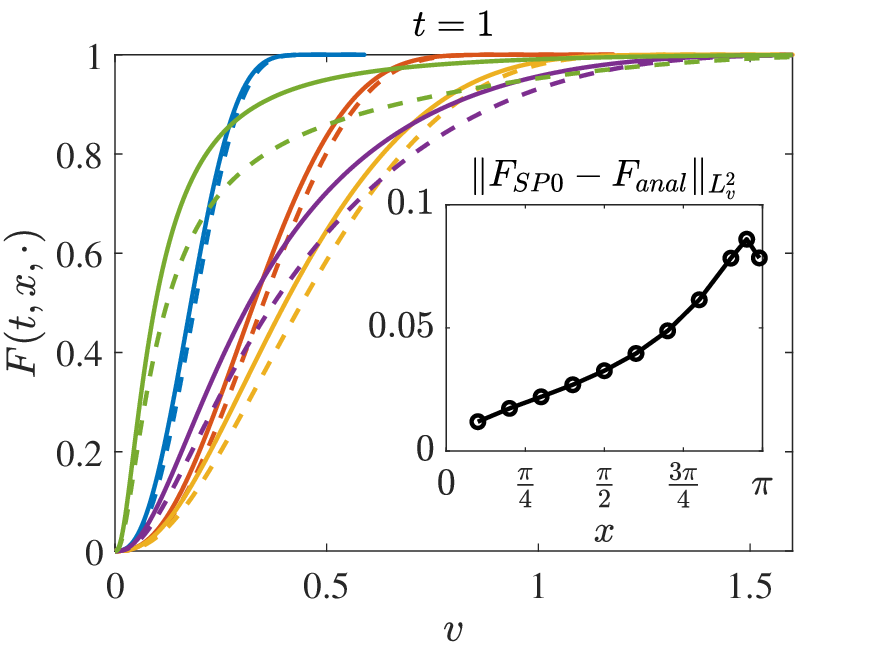}}
  \\
  \subfloat
  {\includegraphics[height=5cm]{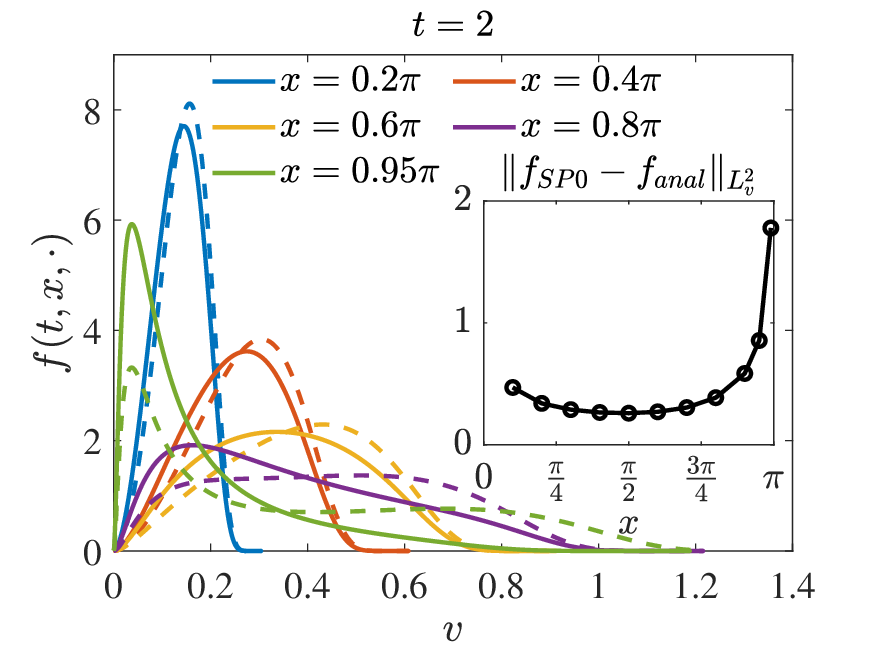}}
  \subfloat
  {\includegraphics[height=5cm]{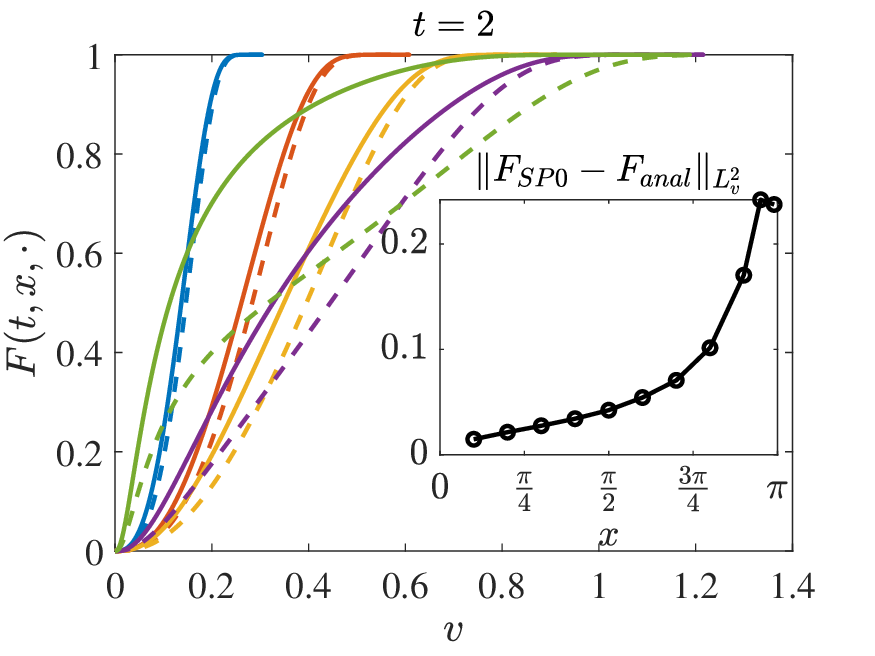}}
  \caption{Case \textbf{II} in (\ref{eq:burgers_2}): Analytical solutions (solid lines) versus `zero-sample' approximations (dashed lines) for $f(t,x,v)$ (left) and $F(t,x,v)$ (right) at $t=1$ (top row) and $t=2$ (bottom row) for different $x$. Insets show the $L_v^2$-errors of $f$ and/or $F$ at $t=1,2$ for different $x$.}
  \label{fig:ctime_f_F_anal}
\end{figure}

\begin{figure}[htbp]
  \centering
  \subfloat
  {\includegraphics[height=5cm]{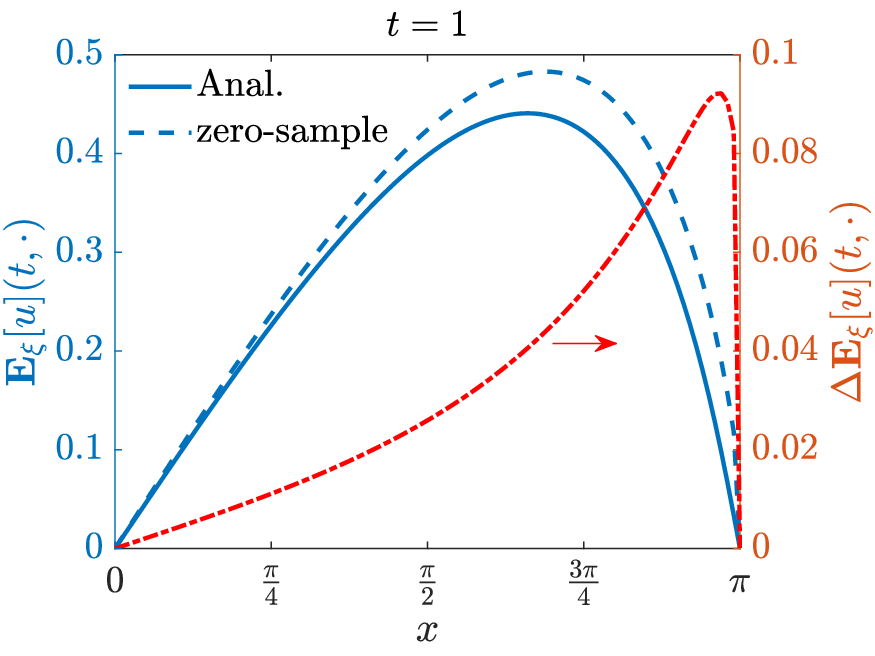}}
  \subfloat
  {\includegraphics[height=5cm]{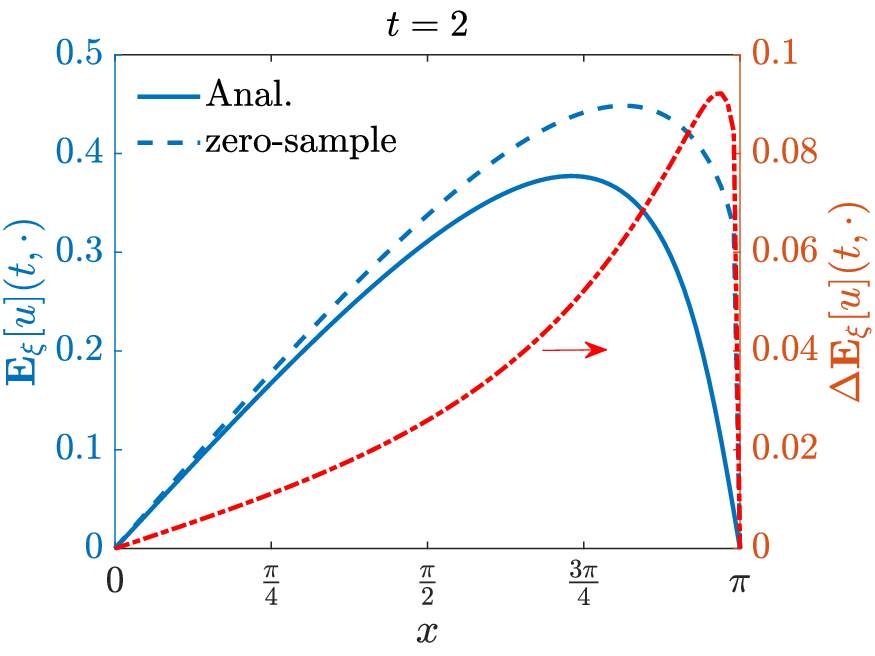}}
  \caption{Expectations $\mathbb E_\xi[u](t,x)$ at $t=1,2$ for Case \textbf{II} in (\ref{eq:burgers_2}). Analytical results (solid lines) are compared with `zero-sample' approximations (dashed lines). Right y-axis: absolute deviations between the two.}
  \label{fig:ctime_Eu}
\end{figure}

We now examine the performance of solving the statistical conservation law (\ref{eq:Fappr}) with approximated $\hat m_N(t,x,v)$. As before, we use a fixed $h_v=0.01$ to construct the estimator. Figures~\ref{fig:ctime_fF} and \ref{fig:dF_c2} present the results at $x=0.95\pi$ and $t=1,2$. With increasing $N$, the approximated PDFs and CDFs clearly converge to the analytical solutions. Quantitatively, the $L_v^2$-errors between the approximated CDFs and the analytical reference decrease significantly as $N$ grows, and remain smaller than those of the empirical CDFs directly reconstructed from the samples $u_i$ ($i=1,\dots,N$) up to $N=5\times10^4$. This finding, consistent with Case \textbf{I}, confirms that the statistical conservation law (\ref{eq:Fappr}) provides an efficient route to reconstructing CDFs from samples.

\begin{figure}[htbp]
  \centering
  \subfloat
  {\includegraphics[height=5cm]{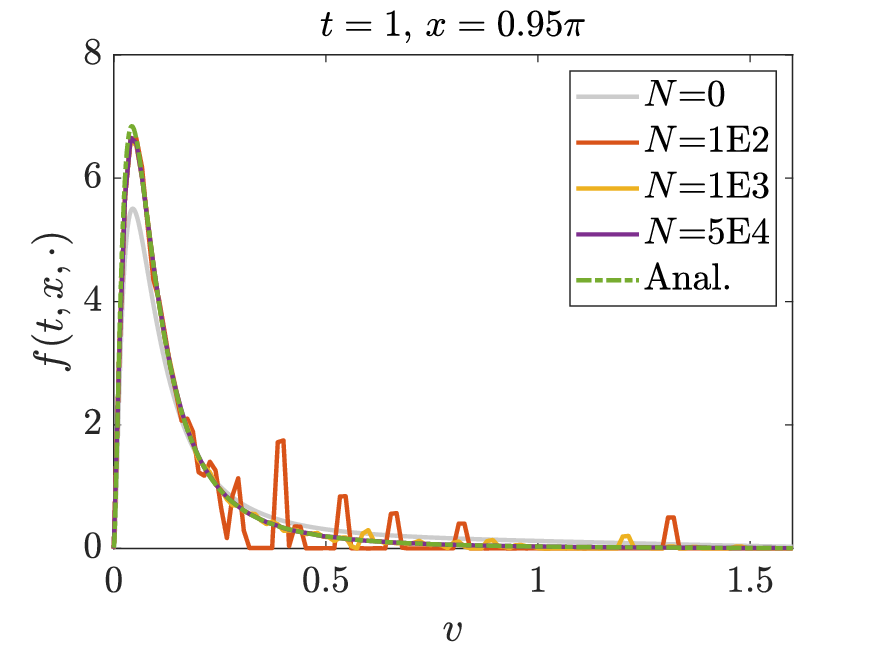}}
  \subfloat
  {\includegraphics[height=5cm]{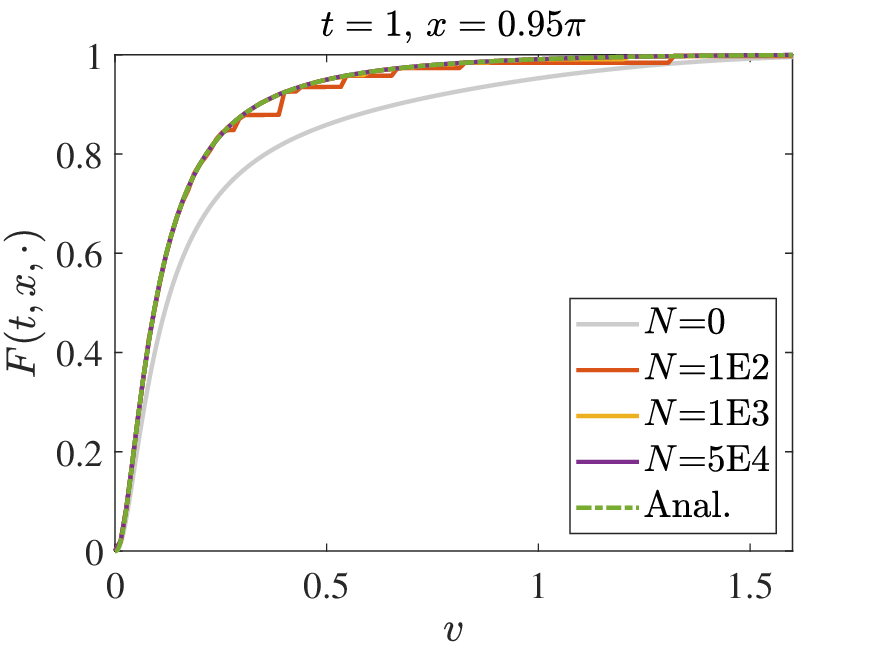}}
  \\
  \subfloat
  {\includegraphics[height=5cm]{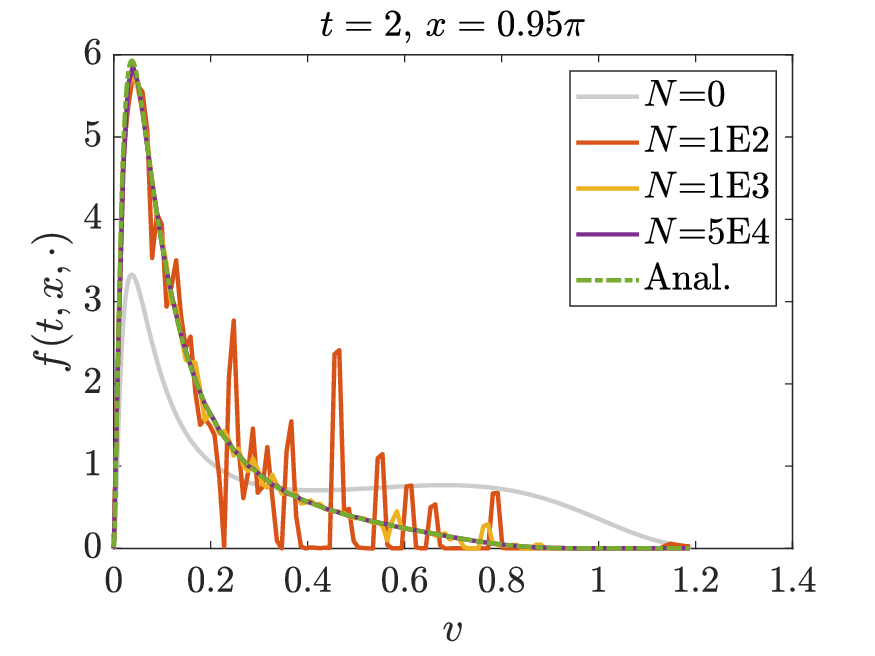}}
  \subfloat
  {\includegraphics[height=5cm]{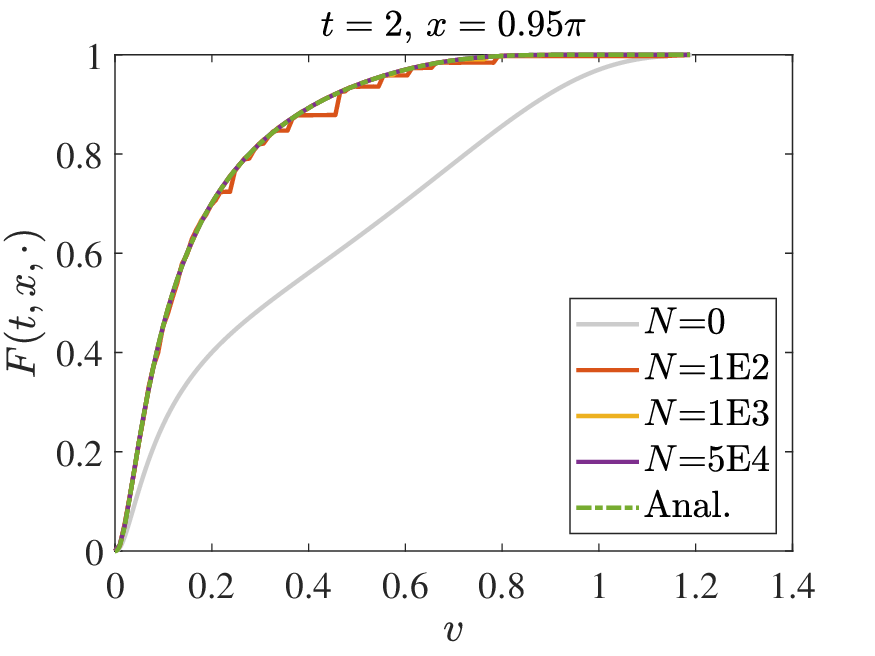}}
  \caption{Approximated PDFs (left) and CDFs (right) at $x=0.95\pi$ and $t=1,2$ for Case \textbf{II} in (\ref{eq:burgers_2}). All results are computed with a fixed $h_v=0.01$ and varying $N$. Here $N=0$ corresponds to the `zero-sample' solutions (\ref{eq:FS0_case2}). Analytical reference solutions are also shown.}
  \label{fig:ctime_fF}
\end{figure}

\begin{figure}[htbp]
  \centering
  \subfloat
  {\includegraphics[height=5cm]{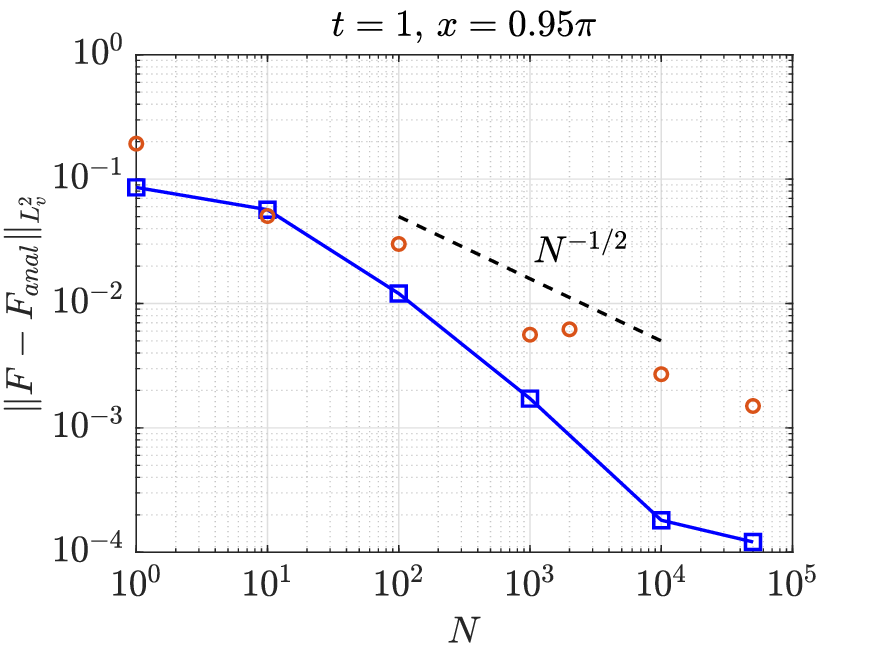}}
  \subfloat
  {\includegraphics[height=5cm]{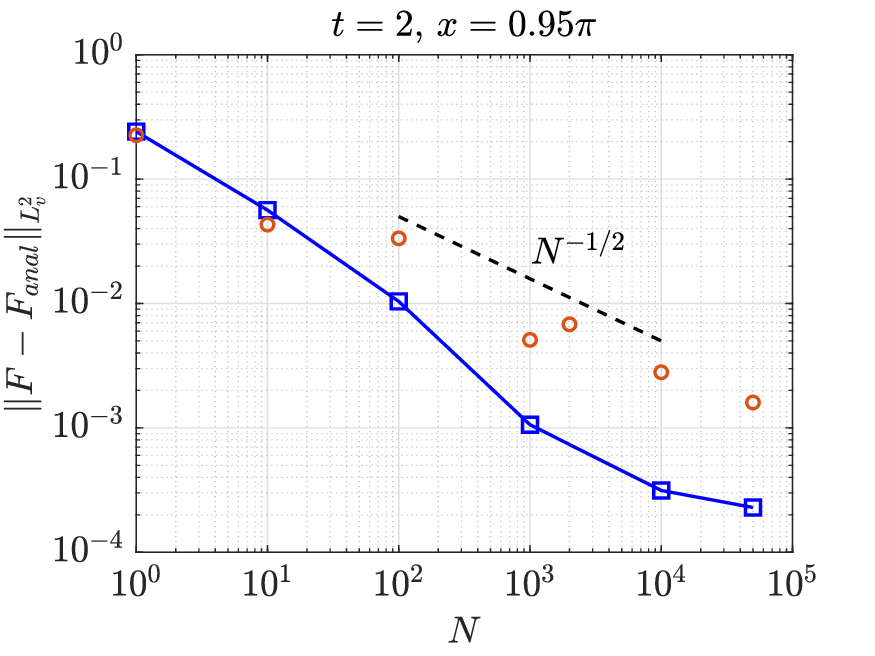}}
  \caption{$L_v^2$-errors between the approximated CDFs and analytical results from (\ref{eq:F_anal}) at $x=0.95\pi$ and $t=1,2$ for Case \textbf{II} in (\ref{eq:burgers_2}). All results are computed with a fixed $h_v=0.01$. Red dots denote the $L^2_v$-errors of the CDFs directly reconstructed from the samples $u_i$ ($i=1,\dots,N$).}
  \label{fig:dF_c2}
\end{figure}

\section{Conclusions} \label{sec:conclusion}

We have developed a statistical framework for viscous balance laws with random initial data. In the scalar setting, we derived statistical conservation laws governing the evolution of (multi-point) probability density functions (PDFs) $f=f(t,x,v)$ and cumulative distribution functions (CDFs) $F=F(t,x,v)$ of the solution. These laws take the form of linear but \emph{unclosed} kinetic-type transport equations in the extended $(t,x,v)$-space (see (\ref{eq:f_eq}) and (\ref{eq:F_eq})). A drift coefficient $m(t,x,v)$, encoding conditional expectations of the viscous contribution in the original balance law, emerges as the key closure term. Notably, the statistical conservation law exhibits a dissipative anomaly: even as viscosity vanishes, the $m$-term persists because it is essential for maintaining the nonnegativity of PDFs.

For numerical computation, we introduced a sampling-based estimator $\hat m_N(t,x,v)$ constructed from realizations of the initial data. This yields a `closed’ equation that allows direct approximation of PDFs/CDFs. For single-point CDFs, we established an a-priori error bound showing that, over finite time, the approximation error is controlled by the estimation error $\|m-\hat m_N\|$. To validate the method, we designed two analytically solvable examples where the solution $u(t,x,\xi)$ remains monotone in the random parameter $\xi$. Numerical tests confirm that the error $\|(F-F_{\mathrm{anal}})(t,x,\cdot)\|$ decreases as the number of samples grows, and is consistently smaller than the error of CDFs reconstructed directly from samples. These results demonstrate that the statistical conservation law framework provides a robust and efficient approach for uncertainty propagation and correlation analysis in viscous balance laws.

Future works may extend the framework to multi-point correlations in both scalar and Navier-Stokes systems, which are fundamental to statistical turbulence theory.

\section*{Acknowledgments}
This work is funded by the Deutsche Forschungsgemeinschaft (DFG, German Research Foundation) - SPP 2410 \textit{Hyperbolic Balance Laws in Fluid Mechanics: Complexity, Scales, Randomness} (CoScaRa). We thank Wen-An Yong (Tsinghua), Ruixi Zhang (Tsinghua) and Martin Oberlack (TU Darmstadt) for the fruitful discussions.

\appendix

\section{Statistical conservation laws for 1D hyperbolic systems} \label{app:sys}

The 1D hyperbolic conservation system for $u=u(t,x,\xi)\in \mathbb R^n$ reads as:
\begin{equation} \label{eq:system}
\begin{split}
  & \partial_tu + \partial_x g(u) = \partial_x(\bm{\epsilon} \partial_x u), \\
  & u(0,x,\xi(\omega)) = u_0(x,\xi(\omega)),
\end{split}
\quad t>0, x\in \mathbb R, \omega\in \Omega,
\end{equation}
where $g:\mathbb R^n \to \mathbb R^n$ is a smooth flux and $\bm{\epsilon}$ is an $n\times n$ diffusion matrix. Following the procedure outlined in Section~\ref{sec:derivation}, we formally obtain the evolution equations for the $K$-point PDF of $u$, $f^{(K)}=f^{(K)}(t,x_1,v_1,\cdots,x_K,v_K)$, given by
\begin{equation}
\left\{
\begin{aligned}
  &\partial_t f^{(K)} + \sum\nolimits_{i=1}^K \nabla_{v_i}\cdot \left( Dg(v_i) \mathcal B_i \right) + \sum\nolimits_{i=1}^K \nabla_{v_i} \cdot \left(m_i f^{(K)} \right) = 0, \\
  &\nabla_{v_i}\cdot \mathcal B_i(t,x_1,v_1,\dots,x_K,v_K) = -\partial_{x_i}f^{(K)}, \quad i=1,\dots,K, \\
  &m_i = m_i(t,x_1,v_1,\dots,x_K,v_K):= \mathbb E_\xi \left[ \partial_{x_i}(\bm{\epsilon} \partial_{x_i} u)|u(t,x_j,\xi) = v_j, \text{ for all } j \right],
\end{aligned}
\right.
\end{equation}
where $\mathcal B_i = \mathcal B_i(t,x_1,v_1,\dots,x_K,v_K) = \mathbb E_\xi \Big[(\partial_{x_i}u) \prod_{j=1}^K \delta(u(t,x_j,\xi)-v_j) \Big]$, and $Dg$ denotes the Jacobian of $g$.

For $K=1$, the above system reduces to the single-point PDF $f=f(t,x,v)$, with $x\in\mathbb R$ and $v\in\mathbb R^n$, governed by
\begin{equation}
\left\{
\begin{aligned}
  &\partial_t f + \nabla_{v}\cdot \left( Dg(v) \mathcal B(t,x,v) \right) + \nabla_{v} \cdot (mf) = 0, \\
  &\nabla_{v}\cdot \mathcal B(t,x,v) = -\partial_{x}f, \\
  &\mathcal B(t,x,v):=\mathbb E_\xi[(\partial_x u)\delta(u(t,x,\xi)-v)], \\
  &m = m(t,x,v):= \mathbb E_\xi \left[ \partial_{x}(\bm{\epsilon} \partial_{x} u)|u(t,x,\xi) = v \right].
\end{aligned}
\right.
\end{equation}
Moreover, additional relations hold between $f$ and $\mathcal B$. Writing $u=(u^j)_{1}^n$, $v=(v^j)_{1}^n$, and $\mathcal B=(B^j)_{1}^n$, we define the marginal PDFs $f^j=f^j(t,x,v^j) = \mathbb E_\xi [\delta(u^j-v^j)] = \int_{\mathbb R^{n-1}} f \prod_{i\ne j}dv^i$. Then by (\ref{eq:chain}) we obtain $\partial_x f^j = -\partial_{v^j} \mathbb E_\xi [\partial_x u^j \delta(u^j-v^j)] = -\partial_{v^j} \int_{\mathbb R^{n-1}} B^j \prod_{i\ne j}dv^i$, which gives
\begin{equation}
  \int_{\mathbb R^{n-1}} \left(B^j + \int^{v_j} \partial_x f(\cdot,v^1,\dots,\tilde v^j,\dots,v^n)d\tilde v^j \right) \prod_{i\ne j}dv^i = 0
\end{equation}
for each $j=1,\dots,n$. These relations indicate the challenges of extending statistical conservation laws to hyperbolic systems.

\bibliographystyle{amsplain}
\bibliography{references}

\end{document}